\documentclass[11pt,amssymb]{amsart}
\usepackage{amssymb}
\usepackage[mathscr]{eucal}

\newcommand{\Z}{{\mathbb Z}}

\newcommand{\C}{{\mathbb C}}
\newcommand{\Q}{{\mathbb Q}}

\newcommand{\tr}{\mathrm{tr}\:}

\newcommand{\R}{{\mathbb R}}

\newcommand{\Br}{\mathrm{Br}}
\newcommand{\Cor}{\mathrm{cor}}

\newcommand{\cO}{\mathscr{O}}

\newcommand{\Ga}{\mathrm{Gal}}
\newtheorem{thm}{Theorem}[section]
\newtheorem{lemma}[thm]{Lemma}
\newtheorem{prop}[thm]{Proposition}
\newtheorem{cor}[thm]{Corollary}

\newcommand{\gen}{\mathbf{gen}}

\newcommand{\Ram}{\mathrm{Ram}}

\begin{document}

\title[Genus and unramified Brauer group]{The genus of a division algebra and the unramified Brauer
group}

\author[V.~Chernousov]{Vladimir I. Chernousov}
\author[A.~Rapinchuk]{Andrei S. Rapinchuk}
\author[I.~Rapinchuk]{Igor A. Rapinchuk}

\address{Department of Mathematics, University of Alberta, Edmonton, Alberta T6G 2G1, Canada}

\email{vladimir@ualberta.ca}

\address{Department of Mathematics, University of Virginia,
Charlottesville, VA 22904-4137, USA}

\email{asr3x@virginia.edu}

\address{Department of Mathematics, Yale University, New Haven, CT
06520-8283, USA}

\email{igor.rapinchuk@yale.edu}

\begin{abstract}
Let $D$ be a finite-dimensional central division algebra over a
field $K$. We define the genus $\gen(D)$ of $D$ to be the collection of
classes $[D'] \in \Br(K)$, where $D'$ is a central division
$K$-algebra having the same maximal subfields as $D$. In this paper,
we describe a general approach to proving the finiteness of
$\gen(D)$ and estimating its size that involves the unramified
Brauer group with respect to an appropriate set of discrete
valuations of $K$. This approach is then implemented in some
concrete situations, yielding in particular an extension of the
Stability Theorem of \cite{RR} from quaternion algebras to arbitrary
algebras of exponent two. We also consider an example where the size
of the genus can be estimated explicitly. Finally, we offer two
generalizations of the genus problem for division algebras: one
deals with absolutely almost simple algebraic $K$-groups having the
same isomorphism/isogeny classes of maximal $K$-tori, and the other
with the analysis of weakly commensurable Zariski-dense subgroups.
\end{abstract}

\maketitle

\section{Introduction}\label{S:Intro}

For a finite-dimensional central division algebra $A$ over a field
$K$, we let $[A]$ denote the corresponding class in the Brauer group
$\Br(K)$ of $K$. Following \cite{CRR}, we define the genus $\gen(D)$
of a central division $K$-algebra $D$ of degree $n$ to be the
collection of all classes $[D'] \in \Br(K)$, where $D'$ is a central
division $K$-algebra having the same maximal fields as $D$ (in
precise terms, this means that $D'$ has the same degree $n$, and a
field extension $P/K$ of degree $n$ admits a $K$-embedding $P
\hookrightarrow D$ if and only if it admits a $K$-embedding $P
\hookrightarrow D'$). One of the results announced in \cite{CRR}
states that if $K$ is a finitely generated field, then
the genus $\gen(D)$
of a central division $K$-algebra $D$ of degree $n$ prime to
$\mathrm{char} \: K$ is finite. The proof consists of two parts:
first, one relates the size of $\gen(D)$ to that of ${}_n\Br(K)_V$,
the $n$-torsion of the unramified Brauer group $\Br(K)_V$ with
respect to a suitable set $V$ of discrete valuations of $K$; second,
one establishes the finiteness of ${}_n\Br(K)_V$. The goal of the
current paper is to give a detailed exposition of the first part.
This analysis, in particular, enables us to extend the Stability
Theorem of \cite{RR} from quaternion algebras to arbitrary algebras
of exponent two. In \cite{CRR}, we sketched a proof, communicated to
us by J.-L.~Colliot-Th\'el\`ene \cite{CT}, of the finiteness of
${}_n\Br(K)_V$ for a suitable $V$, which relies on Deligne's
finiteness theorem for constructible sheaves \cite{Del} and Gabber's
purity theorem \cite{Fuj}.
Our original proof was based on an analysis of the standard exact
sequence for the Brauer group of a curve (cf. \cite{L} or
\cite[(9.25) on p. 27]{GMS}), and the details of this proof will be
given elsewhere. A noteworthy feature of the second proof is that it
leads to explicit estimates on the order of the $n$-torsion of the
unramified Brauer group, hence on the size of the genus: to
demonstrate this point, as well as to showcase some of the ideas
involved in the general argument, we compute an upper bound for the
size of the 2-torsion in the unramified Brauer group of the field of
rational functions of a split elliptic curve over a number field.
In any case, the set of valuations $V$ for which
one can prove the finiteness of ${}_n\Br(K)_V$ is rather special and
arises from geometric considerations; at the same time, one can
relate the size of $\gen(D)$ to that of ${}_n\Br(K)_V$ in a much more
general context (which is our main motivation for separating the two parts of the argument).
So, we begin with a precise description
of the set-up that will be used throughout this paper.

\vskip2mm

Let $K$ be a field. Given a discrete valuation $v$ of $K$, we will
denote by $\mathcal{O}_{K,v}$ and $\overline{K}_v$ its valuation
ring and residue field, respectively. Fix an integer $n > 1$ (which
will later be either the degree or the exponent of $D$) and suppose
that $V$ is a set of discrete valuations of $K$ that satisfies the
following three conditions:

\vskip2mm

\noindent (A)\: \parbox[t]{11.5cm}{\it For any $a \in K^{\times}$,
the set $V(a) := \{ v \in V \: \vert \: v(a) \neq 0\}$ is finite;}

\vskip1mm

\noindent (B)\: \parbox[t]{11.5cm}{\it There exists a \emph{finite}
subset $V' \subset V$ such that the field of fractions of $$\cO :=
\bigcap_{v \in V \setminus V'} \cO_{K , v},$$  coincides with $K$;}

\vskip1mm

\noindent (C)\: \parbox[t]{15cm}{\it For any $v \in V$, the
characteristic of the residue field $\overline{K}_v$ is prime to
$n$.}

\vskip2mm

\noindent (We note that if $K$ is finitely generated, which will be
the case in most of our applications, then  (B) automatically
follows from (A) - see \S \ref{S:Ram}.) Due to (C), we can define
for each $v \in V$ the corresponding {\it residue map}
$$
\rho_v \colon {}_n\Br(K) \longrightarrow
\mathrm{Hom}(\mathcal{G}^{(v)} \: , \: \Z/n\Z),
$$
where $\mathcal{G}^{(v)}$ is the absolute Galois group of
$\overline{K}_v$ (cf., for example, \cite[\S 10]{Salt} or
\cite[Ch.II, Appendix]{Serre}). As usual, a class $[A] \in
{}_n\Br(K)$ (or a finite-dimensional central simple $K$-algebra $A$
representing this class) is said to be {\it unramified} at $v$ if
$\rho_v([A]) = 1$, and {\it ramified} otherwise. We let $\Ram_V(A)$
(or $\Ram_V([A])$) denote the set of all $v \in V$ where $A$ is
ramified; one shows that this set is always finite (Proposition
\ref{P:R1}). We also define the unramified part of ${}_n\Br(K)$ with
respect to $V$ to be
$$
{}_n\Br(K)_V = \bigcap_{v \in V} \mathrm{Ker} \: \rho_v.
$$
The goal of \S \ref{S:Ram} is to prove the following result that
relates the size of the genus with the order of the unramified Brauer
group.

\vskip2mm

\noindent {\bf Theorem \ref{T:Ram1}.} {\it Assume that
${}_n\Br(K)_V$ is finite. Then for any finite-dimensional central
division $K$-algebra $D$ of exponent $n$, the intersection $\gen(D)
\cap {}_n\Br(K)$ is finite,  of size
$$
\vert \gen(D) \cap {}_n\Br(K) \vert \leqslant \vert {}_n\Br(K)_V
\vert \cdot \varphi(n)^r, \ \ \text{with} \ \ r = \vert \Ram_V(D)
\vert,
$$
where $\varphi$ is the Euler function. In particular, if $D$ has
degree $n$ then $$\vert \gen(D) \vert \leqslant \vert {}_n\Br(K)_V
\vert \cdot \varphi(n)^r.$$}

\vskip2mm

We use this result in \S\ref{S:Stab} to estimate the size of the
genus for division algebras over the function fields of curves in
certain situations. This analysis, in particular, enables us to
describe some cases where $\gen(D)$ reduces to a single
element. First, we observe that since the opposite algebra
$D^{\mathrm{op}}$ has the same maximal subfields as $D$, this can
happen only if $[D^{\mathrm{op}}] = [D]$, i.e. if $D$ has exponent 2
in the Brauer group. On the other hand, it follows from the theorem
of Artin-Hasse-Brauer-Noether (AHBN) (cf. 3.6) that $\gen(D)$ does
reduce to a single element for any algebra $D$ of exponent 2 over a
global field $K$ (in which case $D$ is necessarily a quaternion
algebra).
The following theorem, which was established earlier in \cite{RR}
for quaternion algebras, expands the class of fields with this
property.

\vskip2mm

\noindent {\bf Theorem \ref{T:Stab1}.} (Stability Theorem) {\it Let
$k$ be a field of characteristic $\neq 2$.

\vskip2mm

\noindent {\rm (1)} Suppose $k$ satisfies the following property:

\vskip2mm

\noindent $(*)$ \parbox[t]{16cm}{If $D$ and $D'$ are central
division $k$-algebras of exponent 2 having the same maximal
subfields, then $D \simeq D'$ $($in other words, for any $D$ of
exponent 2, $\vert \gen(D) \cap {}_2\Br(k) \vert = 1)$.}

\vskip2mm

\noindent Then the field of rational functions $k(x)$ also satisfies
$(*)$.

\vskip2mm

\noindent {\rm (2)} \parbox[t]{16cm}{If $\vert \gen(D) \vert = 1$
for any central division $k$-algebra $D$ of exponent 2, then the
same is true for any central division $k(x)$-algebra of exponent
2.}}

\vskip3mm

\noindent {\bf Corollary \ref{C:Stab1}.} {\it Let $k$ be either a
finite field of characteristic $\neq 2$ or a number field, and $K =
k(x_1, \ldots , x_r)$ be a finitely generated purely transcendental
extension of $k$. Then for any central division $K$-algebra $D$ of
exponent 2, we have $\vert \gen(D) \vert = 1$.}

\vskip3mm

In \S\ref{S:Ell}, we will give explicit estimates on the size of the
genus of a quaternion algebra over the field of rational functions
of a split elliptic curve over a number field (cf. Theorem
\ref{T:Elliptic1}). Finally, in \S\S\ref{S:Gen}-\ref{S:Geom} we
discuss possible generalizations of the finiteness theorem for the
genus \cite{CRR} in the context of general absolutely almost simple
algebraic groups. More precisely, Conjecture 5.1 predicts the
finiteness of the genus, defined in terms of the isomorphism classes
of maximal $K$-tori, of an absolutely almost simple algebraic group
over a finitely generated field $K$ of ``good" characteristic --
Theorem \ref{T:Al} confirms this conjecture for inner forms of type
$\textsf{A}_{\ell}$. In \S\ref{S:Geom}, after a brief review of the
notion of weak commensurability and its connections with
length-commensurability of locally symmetric spaces (cf. \cite{PR1},
\cite{PR-Gen}), we formulate Conjecture 6.1 that asserts the
finiteness of the number of forms of a given absolutely simple
algebraic group over a finitely generated field $K$ of
characteristic zero that can contain a finitely generated
Zariski-dense subgroup with the trace field $K$ weakly commensurable
to a given finitely generated Zariski-dense subgroup (both
Conjecture 5.1 and 6.1 are true over number fields).

\section{Ramification places and the genus of a division algebra}\label{S:Ram}

Let $K$ be a field. Fix an integer $n > 1$ and let $V$ be a set of
discrete valuations of $K$ satisfying conditions (A), (B) and (C) of
\S\ref{S:Intro}. (We observe that if $K$ is generated over its prime
subfield by nonzero elements $a_1, \ldots , a_r,$ then using (A), one
can find a finite subset $V' \subset V$ such that $v(a_i) = 0$ for
all $v \in V \setminus V'$ and all $i = 1, \ldots , r$. Then $a_1,
\ldots , a_r$ lie in $\cO = \bigcap_{v \in V \setminus V'}
\cO_{K,v}$, and hence the fraction field of the latter coincides
with $K$. Thus, for a finitely generated field $K$, condition (B)
follows automatically from condition (A).)

\vskip2mm

\begin{prop}\label{P:R1}
Assume that $V$ satisfies conditions {\rm (A)}, {\rm (B)}, and {\rm
(C)}. Then for any $[A] \in {}_n\Br(K)$, the set $\Ram_V([A])$ is finite.
\end{prop}
\begin{proof}
Pick a finite set $V' \subset V$ as in (B), and set
$$
\cO = \bigcap_{v \in V \setminus V'} \cO_{K , v}.
$$
It is enough to show that the set of $v \in V \setminus V'$ where
$A$ ramifies is finite. Let $\dim_K A = \ell^2$. First, we note
that it is possible to find a basis $x_1 = 1, \ldots , x_{\ell^2}$
of $A$ over $K$ such that
$$
{\mathcal A} := \cO x_1 + \cdots + \cO x_{\ell^2}
$$
is a subring of $A$. Indeed, let $y_1, \ldots , y_{\ell^2}$ be an
arbitrary $K$-basis of $A$ with $y_1 = 1$. Then, we can write
\begin{equation}\label{E:R1}
y_iy_j = \sum_{k = 1}^{\ell^2} c^k_{ij} y_k \ \ \ \text{with} \ \
c^k_{ij} \in K.
\end{equation}
Since $K$ is the field of fractions of $\cO$, there exists $d \in \cO$
such that $d c^k_{ij} \in \cO$ for all $i, j$ and $k$. Multiplying
(\ref{E:R1}) by $d^2$, we obtain
$$
(dy_i)(dy_j) = \sum_{k = 1}^{\ell^2} (dc^k_{ij}) (d y_k),
$$
which implies that the basis $x_1 = 1, x_2 = dy_2, \ldots ,
x_{\ell^2} = d y_{\ell^2}$ is as required.

Now, for $v \in V \setminus V'$,
we set $$A_v = A \otimes_K K_v \ \ \text{and} \ \  \mathcal{A}_v =
\mathcal{A} \otimes_{\mathcal{O}} \mathcal{O}_v.$$ We have $A_v =
\mathcal{A}_v \otimes_{\mathcal{O}_v} K_v$, and furthermore if
$\mathcal{A}_v$ is an Azumaya algebra, then $A$ is unramified at $v$
(cf. \cite[\S 10]{Salt}). Since $A$ is a central simple algebra over
$K$, the canonical map $\varphi \colon A \otimes_K A^{\mathrm{op}}
\to \mathrm{End}_K\: A$  given by $$\sum_{i = 1}^r a_i \otimes b_i
\longrightarrow \left( x \mapsto \sum_{i = 1}^r a_ixb_i \right),$$
is an isomorphism. Identifying $\mathcal{A} \otimes_{\mathcal{O}}
\mathcal{A}^{\mathrm{op}}$ and $\mathrm{End}_{\mathcal{O}}\:
\mathcal{A}$ with $\mathcal{O}$-submodules of $A \otimes_K
A^{\mathrm{op}}$ and $\mathrm{End}_K A$, respectively, we observe
that the fact that $\varphi$ is an isomorphism, in conjunction with
the finite generation of the $\mathcal{O}$-module
$\mathrm{End}_{\mathcal{O}}\: \mathcal{A} \simeq
M_{\ell^2}(\mathcal{O})$, implies the existence of a nonzero $t \in
\mathcal{O}$ such that
\begin{equation}\label{E:t}
t \cdot \mathrm{End}_{\mathcal{O}}\: \mathcal{A} \subset
\varphi(\mathcal{A} \otimes_{\mathcal{O}}
\mathcal{A}^{\mathrm{op}}).
\end{equation}
Now, suppose $v \in V$ lies outside of the finite set $V' \cup
V(t)$, and let $\varphi_v$ be the map analogous to $\varphi$ for the
algebra $A_v$. Since $\mathrm{End}_{\mathcal{O}_v}\: \mathcal{A}_v =
(\mathrm{End}_{\mathcal{O}}\: \mathcal{A}) \otimes_{\mathcal{O}}
\mathcal{O}_v$ and $t \in \mathcal{O}_v^{\times}$, we conclude from
(\ref{E:t}) that
$$
\varphi_v(\mathcal{A}_v \otimes_{\mathcal{O}_v}
\mathcal{A}_v^{\mathrm{op}}) = \mathrm{End}_{\mathcal{O}_v}\:
\mathcal{A}_v.
$$
This means that $\mathcal{A}_v$ is an Azumaya $\mathcal{O}_v$-algebra, from which
our claim follows.
\end{proof}

\vskip2mm

The main result of this section is the following.
\begin{thm}\label{T:Ram1}
Assume that ${}_n\Br(K)_V$ is finite. Then for any
finite-dimensional central division $K$-algebra $D$ of exponent $n$,
the intersection $\gen(D) \cap {}_n\Br(K)$ is finite,  of size
$$
\vert \gen(D) \cap {}_n\Br(K) \vert \leqslant \vert {}_n\Br(K)_V
\vert \cdot \varphi(n)^r, \ \ \text{with} \ \ r = \vert \Ram_V(D)
\vert,
$$
where $\varphi$ is the Euler function. In particular, if $D$ has
degree $n$ then $$\vert \gen(D) \vert \leqslant \vert {}_n\Br(K)_V
\vert \cdot \varphi(n)^r.$$
\end{thm}

\vskip2mm

We begin the proof with the following generalization of Lemma 2.5 of
\cite{RR}, which establishes the desired conclusion without the
assumption that the residue field satisfies condition (LD)
introduced in {\it loc. cit.} We recall that given a
finite-dimensional central division algebra $\mathcal{D}$ over a
field $\mathcal{K}$ which is complete with respect to a discrete
valuation $v$, the valuation $v$ uniquely extends to a discrete
valuation $\tilde{v}$ of $\mathcal{D}$ (cf. \cite[Ch. XII, \S2]{Se},
\cite{W1}). Furthermore, the corresponding valuation ring
$\cO_{\mathcal{D}}$ has a unique maximal 2-sided ideal
$\mathfrak{P}_{\mathcal{D}}$ (the valuation ideal), and the quotient
$\overline{\mathcal{D}} =
\cO_{\mathcal{D}}/\mathfrak{P}_{\mathcal{D}}$ is a
finite-dimensional division (but not necessarily central) algebra,
called the {\it residue algebra}, over the residue field
$\overline{\mathcal{K}}$.
\begin{lemma}\label{L:R2}
Let $\mathcal{K}$ be a field complete with respect to a discrete
valuation $v$ with residue field $k$. Suppose $\mathcal{D}_1$
and $\mathcal{D}_2$ are two finite-dimensional central division
$\mathcal{K}$-algebras of degree $n$ prime to $\mathrm{char} \: k$,
and, for $i = 1,2$, let $\mathcal{E}_i$ be the center of the residue
algebra $\overline{\mathcal{D}}_i$.
If $\mathcal{D}_1$ and $\mathcal{D}_2$ have the same maximal
subfields, then $\mathcal{E}_1 = \mathcal{E}_2$.
\end{lemma}
\begin{proof}
Recall that $\mathcal{E}_1$ and $\mathcal{E}_2$ are cyclic Galois
extensions of $k$ (cf. \cite[Proposition 2.5]{W2}). By symmetry, to
prove that $\mathcal{E}_1 = \mathcal{E}_2$, it suffices to prove the
existence of a $k$-embedding $\mathcal{E}_1 \hookrightarrow
\mathcal{E}_2$. Assume that there is no such embedding, and let
$\mathcal{L}_i$ denote the unramified extension of $\mathcal{K}$
with residue field $\mathcal{E}_i$. Then $\mathcal{L}_i$ is a cyclic
Galois extension of $\mathcal{K}$, and $\mathcal{L}_1
\not\hookrightarrow \mathcal{L}_2$. By construction, $\mathcal{L}_2$
embeds in $\mathcal{D}_2$ and, since $\mathcal{D}_1$ and
$\mathcal{D}_2$ have the same maximal subfields, $\mathcal{L}_2$
embeds in $\mathcal{D}_1$ as well. Let $\mathcal{L}^{(i)}_2$ be the
image of some $\mathcal{K}$-embedding $\mathcal{L}_2 \hookrightarrow
\mathcal{D}_i$, and let $\Delta_i$ be the centralizer of
$\mathcal{L}^{(i)}_2$ in $\mathcal{D}_i$. It is well-known (and
follows from the proof of the Double Centralizer Theorem) that
$\Delta_i$ is Brauer-equivalent to $\mathcal{D}_i
\otimes_{\mathcal{K}} \mathcal{L}^{(i)}_2$. We now observe that the
assumption that $\mathcal{L}_1 \not\hookrightarrow \mathcal{L}_2$
implies that the $\mathcal{L}_2$-algebra $\mathcal{D}_1
\otimes_{\mathcal{K}} \mathcal{L}^{(1)}_2$ is ramified with respect
to the extension $w$ of $v$ to $\mathcal{L}_2$. To see this, we will
use the following well-known statement.

\begin{thm}\label{T:Ram2}
{\rm (\cite[Theorem 10.4]{Salt})} Let $\mathcal{K}$ be a field
complete with respect to a discrete valuation $v$, with  residue
field $k$, and let $n > 1$ be an integer prime to $\mathrm{char} \:
k$. For a finite extension $\mathcal{L}/\mathcal{K}$, we let $w$ and
$\ell$ denote the extension of $v$ and the corresponding residue
field.
Then the following diagram
$$
\begin{array}{ccc}
{}_n\Br(\mathcal{L}) & \stackrel{\rho_w}{\longrightarrow} &
\mathrm{Hom}(\mathcal{G}^{(w)} , \Z/n\Z) \\
\uparrow & & \uparrow [e] \\
{}_n\Br(\mathcal{K}) & \stackrel{\rho_v}{\longrightarrow} &
\mathrm{Hom}(\mathcal{G}^{(v)} , \Z/n\Z)
\end{array}
$$
in which $\mathcal{G}^{(v)}$ and $\mathcal{G}^{(w)}$ are the
absolute Galois groups of $k$ and $\ell$, respectively, $\rho_v$ and
$\rho_w$ the corresponding residue maps, and $[e]$ denotes the
composition of the natural restriction map with multiplication by
the ramification index $e = e(w \vert v)$, commutes.
\end{thm}

It is well-known that the subgroup of the absolute Galois group
$\mathcal{G}^{(v)}$ of $k$ fixing $\mathcal{E}_1$ coincides with
$\mathrm{Ker} \: \rho_v([\mathcal{D}_1])$ (\cite[Theorem 3.5]{W2}).
So, the assumption that $\mathcal{E}_1 \not\hookrightarrow
\mathcal{E}_2$ means that the restriction of
$\rho_v([\mathcal{D}_1])$ to the subgroup $\mathcal{G}^{(w)}$ of
$\mathcal{G}^{(v)}$ corresponding to $\mathcal{E}_2$, is nontrivial.
Since $e(w \vert v) = 1$, Theorem \ref{T:Ram2} implies that
$\rho_w([\mathcal{D}_1 \otimes_{\mathcal{K}} \mathcal{L}_2])$ is
nontrivial, i.e. $\mathcal{D}_1 \otimes_{\mathcal{K}}
\mathcal{L}^{(1)}_2$ is ramified at $w$, as claimed.

Let $\tilde{w}$ be the extension of $w$ to $\Delta_1$. Since
$[\Delta_1] = [\mathcal{D}_1 \otimes_{\mathcal{K}}
\mathcal{L}^{(1)}_2]$ is ramified at $w$, the ramification index
$e(\tilde{w} \vert w)$ is $> 1$ (cf. \cite[Theorem 3.4]{W2}; note
that being of degree prime to $\mathrm{char} \: k$, the division
algebra $\Delta_1$ is ``inertially split"). It follows that
$\Delta_1$ contains a maximal subfield $\mathcal{P}$ which is
ramified over $\mathcal{L}^{(1)}_2$.  By our assumption,
$\mathcal{P}$  embeds into $\mathcal{D}_2,$ and moreover, by the
Skolem-Noether theorem, we may assume that this embedding maps
$\mathcal{L}^{(1)}_2$ onto $\mathcal{L}^{(2)}_2$, hence its image
(which we will also denote by $\mathcal{P}$) is contained in
$\Delta_2$. Since $\mathcal{P}/\mathcal{L}^{(2)}_2$ is ramified, we
conclude that the $\mathcal{L}_2$-algebra $\Delta_2$, and hence also
$\mathcal{D}_2 \otimes_{\mathcal{K}} \mathcal{L}_2$, is ramified
with respect to $w$. On the other hand, it follows from Theorem
\ref{T:Ram2} that $\mathcal{D}_2 \otimes_{\mathcal{K}}
\mathcal{L}_2$ is unramified at $w$, a contradiction.
\end{proof}

\begin{lemma}\label{L:=}
Let $D$ and $D'$ be central division $K$-algebras such that $[D] \in
{}_n\Br(K)$ and $[D'] \in \gen(D) \cap {}_n\Br(K)$. Given $v \in V$,
we let $\chi_v$ and $\chi_v' \in \mathrm{Hom}(\mathcal{G}^{(v)} ,
\Z/n\Z)$ denote the images of $[D]$ and $[D']$, respectively, under the residue map $\rho_v$.
Then
$$
\mathrm{Ker}\: \chi_v = \mathrm{Ker} \chi_v'
$$
for all $v \in V$.
\end{lemma}
\begin{proof} Write
$$
D \otimes_K K_v = M_{\ell}(\mathcal{D}) \ \ \ \text{and} \ \ \ D'
\otimes_K K_v = M_{\ell'}(\mathcal{D}'),
$$
where $\mathcal{D}$ and $\mathcal{D}'$ are central division algebras
over $\mathcal{K} = K_v$. According to \cite[Corollary 2.4]{RR}, we
have $\ell = \ell'$ and $\mathcal{D}$ and $\mathcal{D}'$ have the
same maximal subfields. Letting $\mathcal{E}$ and $\mathcal{E}'$
denote the centers of the residue algebras $\overline{\mathcal{D}}$
and $\overline{\mathcal{D}}'$, respectively, we infer from Lemma
\ref{L:R2} that $\mathcal{E} = \mathcal{E}'$ (note that the lemma
applies since by assumption $n$ is relatively prime to
$\mathrm{char}\: \overline{K}_v$). On the other hand, as we already
mentioned in the proof of Lemma \ref{L:R2}, $\mathrm{Ker}\: \chi_v$
and $\mathrm{Ker}\: \chi'_v$ are precisely the subgroups of
$\mathcal{G}^{(v)}$ corresponding to $\mathcal{E}$ and
$\mathcal{E}'$, respectively (cf. \cite[Theorem 3.5]{W2}). So, our
claim follows.
\end{proof}

\vskip2mm



\noindent {\it Proof of Theorem \ref{T:Ram1}.} Suppose that $[D']
\in \gen(D) \cap {}_n\Br(K)$. Fix $v \in V$, and set $\chi_v =
\rho_v([D])$ and $\chi'_v = \rho_v([D'])$. According to Lemma
\ref{L:=}, we have
\begin{equation}\label{E:R4}
\mathrm{Ker}\: \chi_v = \mathrm{Ker}\: \chi'_v.
\end{equation}
Let $m$ be the order of $\chi_v$. Any character $\chi'_v$ of
$\mathcal{G}^{(v)}$ satisfying (\ref{E:R4}) can be viewed as a
faithful character of the order $m$ cyclic group
$\mathcal{G}^{(v)}/\mathrm{Ker}\: \chi_v$, and therefore there are
$\varphi(m)$ possibilities for $\chi'$. So,
$$
\vert \rho_v(\gen(D) \cap {}_n\Br(K)) \vert \leqslant \varphi(m)
\leqslant \varphi(n)
$$
for any $v \in V$ (as $m$ divides $n$), and $$\rho_v(\gen(D) \cap
{}_n\Br(K)) = \{1\}$$ if $\rho_v([D]) = 1$.

Now, since, by Proposition \ref{P:R1}, any division algebra ramifies at a finite number of places, we can consider the map
\begin{equation}\label{E:rhoV}
\rho_V \colon {}_n\Br(K) \longrightarrow \bigoplus_{v \in V}
\mathrm{Hom}(\mathcal{G}^{(v)} , \Z/n\Z), \ \ \rho = (\rho_v).
\end{equation}
Our previous discussion shows that
$$
\vert \rho_V(\gen(D) \cap {}_n\Br(K)) \vert \leqslant \varphi(n)^r
$$
where $r = \vert \Ram_V(D) \vert$. By definition $\mathrm{Ker}\:
\rho_V = {}_n\Br(K)_V$, so we obtain the required estimate.

\vskip1mm

Now, if $D$ has degree $n$, then clearly $\gen(D) \subset
{}_n\Br(K)$, and our second assertion follows from the first one.
\hfill $\Box$

\vskip2mm

\noindent {\bf Remark 2.6.} 1. Our proof of Theorem \ref{T:Ram1}
actually leads to the following somewhat stronger assertion, which
will be used in \S \ref{S:Gen}. Let $K$ be a field, and $V$ be a set
of discrete valuations of $K$ satisfying conditions (A), (B) and (C)
for a given integer $n > 1$. For a  central division algebra $D$ of
degree $n$ over $K$, we define the {\it local genus} $\gen_V(D)$ of
$D$ with respect to $V$ as the collection of classes $[D'] \in
\Br(K),$ where $D'$ is a central division algebra $K$-algebra of
degree $n$ such that for any $v \in V$, if one writes $D \otimes_K
K_v = M_{\ell}(\mathcal{D})$ and $D' \otimes_K K_v =
M_{\ell'}(\mathcal{D}')$ where $\mathcal{D}$ and $\mathcal{D}'$ are
central division algebras over $K_v$, then $\ell = \ell'$ and
$\mathcal{D}$ and $\mathcal{D}'$ have the same maximal separable
subfields. If $n$ is prime to $\mathrm{char} \: K$ and
${}_n\Br(K)_V$ is finite, then $\gen_V(D)$ is finite for any central
division $K$-algebra $D$ of degree $n$.

\vskip1mm

2. Lemma \ref{L:R2} remains valid (and so do its consequences) if
one replaces the assumption that the degree $n$ is prime to
$\mathrm{char} \: k$ by the weaker assumption that $\mathcal{E}_1$
and $\mathcal{E}_2$ are separable extensions of $k$. This makes
$\mathcal{D}_1$ and $\mathcal{D}_2$ ``inertially split," and the
argument goes through without any significant changes.

\section{The genus over the function fields of curves}\label{S:Stab}

We will now apply the results of \S\ref{S:Ram} in the case where $K
= k(C)$ is the field of rational functions on a smooth absolutely
irreducible projective curve $C$ over a field $k$, and $V$ is the
set of all geometric places of $K$, i.e. those discrete valuations
of $K$ that are trivial on $k$. If $n > 1$ is an integer prime to
$\mathrm{char} \: k$, then it is clear that $V$ satisfies conditions
(A), (B) and (C). The corresponding unramified Brauer ${}_n\Br(K)_V$
will then, following tradition, be denoted by
${}_n\Br(K)_{\mathrm{ur}}$ (it is known that this is precisely the
$n$-torsion subgroup of the Brauer group of the curve $C$, cf.
\cite{L}). Note that there is a natural map $\iota_k \colon
{}_n\Br(k) \to {}_n\Br(K)_{\mathrm{ur}}$. The following theorem
provides an estimation of the size of $\gen(D) \cap {}_n\Br(K)$ for
a central division $K$-algebra $D$ of exponent $n$ in certain
situations.

%
%
%
%

\begin{thm}\label{T:4}
Let $n > 1$ be an integer prime to $\mathrm{char}\: k$. Assume that

\vskip2mm

\noindent $\bullet$ the set $C(k)$ of rational points is infinite;

\vskip1mm

\noindent $\bullet$ \parbox[t]{16cm}{$\vert
{}_n\Br(K)_{\mathrm{ur}}/\iota_k({}_n\Br(k)) \vert =: M < \infty$.}

\vskip2mm

\noindent Then

\vskip2mm

\noindent {\rm (1)} \parbox[t]{16cm}{if there exists $N < \infty$
such that $$\vert \gen(\Delta) \cap {}_n\Br(k) \vert \leqslant N$$
for any central division $k$-algebra $\Delta$ of exponent $n$, then
for any central division $K$-algebra $D$ of exponent $n$ we have
$$\vert \gen(D) \cap {}_n\Br(K) \vert \leqslant M \cdot N \cdot
\varphi(n)^r,$$ where $r = \vert \mathrm{Ram}_V(D) \vert$;}

\vskip2mm

\noindent {\rm (2)} \parbox[t]{16cm}{if $\gen(\Delta) \cap
{}_n\Br(k)$ is \emph{finite} for any central division $k$-algebra
$\Delta$ of exponent $n$, then $\gen(D)~\cap~{}_n\Br(K)$ is
\emph{finite} for any central division $K$-algebra $D$ of exponent
$n$.}
\end{thm}

\begin{proof}
Let $D$ be a finite-dimensional central division $K$-algebra such
that $[D] \in {}_n\Br(K)$, and set $r = \vert \mathrm{Ram}_{V}(D)
\vert.$ Arguing as in the proof of Theorem \ref{T:Ram1}, we see that
in the notations introduced therein, we have
$$\vert \rho_V(\gen(D) \cap {}_n\Br(K)) \vert \leqslant
\varphi(n)^r,$$ i.e.,  $\gen(D) \cap {}_n\Br(K)$ is contained in a
union of $\leqslant \varphi(n)^r$ cosets modulo
${}_n\Br(K)_{\mathrm{ur}}$.
It follows that $\gen(D) \cap {}_n\Br(K)$ is contained in a union of
$\leqslant M \cdot \varphi(n)^r$ cosets modulo
$\iota_k({}_n\Br(k))$. Thus, it is enough to show that for $[D'] \in
\gen(D)$ of exponent $n$, the intersection $\gen(D) \cap ([D'] \cdot
\iota_k({}_n\Br(k)))$ is

\vskip2.5mm

\noindent $\bullet$ \parbox[t]{16cm}{{\it finite} if $\gen(\Delta)
\cap {}_n\Br(k)$ is finite for any central division $k$-algebra
$\Delta$ with $[\Delta] \in {}_n\Br(k)$;}

\vskip1.5mm

\noindent $\bullet$ \parbox[t]{16cm}{{\it of size} $\leqslant N$ if
$N < \infty$ has the property that $\vert \gen(\Delta) \cap
{}_n\Br(k) \vert \leqslant N$ for any $\Delta$ as above.}

\vskip2.5mm


\noindent Notice that for $[D'] \in \gen(D)$, we have $\gen(D') =
\gen(D)$, so, to simplify our notations, we may replace $D'$ by $D$.
Then, our problem reduces to proving the above two statements for
the intersection $$\Lambda := \gen(D) \cap ([D] \cdot
\iota_k({}_n\Br(k))),$$ where $D$ is any finite-dimensional central
division $K$-algebra such that $[D] \in {}_n\Br(K)$.

\vskip2mm

For this, we first recall (cf., for example, \cite[Ch. XII, \S
3]{Se} or \cite[\S3]{W2}) that given a field $\mathcal{K}$ complete
with respect to a discrete valuation $v$ and an integer $n
> 1$ prime to the characteristic of the residue field
$\overline{\mathcal{K}}$, there is a natural isomorphism $\nu$
between the unramified Brauer group ${}_n\Br(\mathcal{K})_{\{v\}}$
and ${}_n\Br(\overline{\mathcal{K}})$. This isomorphism can be
described as follows: if $[\mathcal{D}] \in {}_n\Br(\mathcal{K})_{\{
v \}}$ is represented by a central division $\mathcal{K}$-algebra
$\mathcal{D}$, then the residue division algebra
$\overline{\mathcal{D}}$ is central over $\overline{\mathcal{K}}$ and
$\nu([\mathcal{D}]) = [\overline{\mathcal{D}}]$.

\vskip1mm

Now let $D$ be a finite-dimensional central division algebra over
$K$ such that $[D] \in {}_n\Br(K)$. Since $C(k)$ is infinite, we can
pick $v \in V$ such that $\overline{K}_v = k$ and $D$ is unramified
at $v.$ Let $\mathcal{K} = K_v$ be the completion of $K$ with
respect to $v$, and define $\nu_v \colon {}_n\Br(K)_{\{v\}} \to
{}_n\Br(k)$ to be the composition of the natural map
${}_n\Br(K)_{\{v\}} \to {}_n\Br(\mathcal{K})_{\{v\}}$ with the
isomorphism $\nu \colon {}_n\Br(\mathcal{K})_{\{v\}} \to {}_n\Br(k)$
constructed above (note that $\overline{\mathcal{K}} = k$). Pick any
$[D'] \in \gen(D) \cap {}_n\Br(K)$, and write
$$
D \otimes_K \mathcal{K} = M_{\ell}(\mathcal{D}) \ \ \ \text{and} \ \
\ D' \otimes_K \mathcal{K} = M_{\ell'}(\mathcal{D}'),
$$
with $\mathcal{D}$ and $\mathcal{D}'$  division
$\mathcal{K}$-algebras. As in the proof of Lemma \ref{L:=}, we infer
from \cite[Corollary 2.4]{RR} that $\ell = \ell'$ and $\mathcal{D}$
and $\mathcal{D}'$ have the same maximal subfields. Furthermore, it
follows from Lemma \ref{L:=} that $D'$ (equivalently,
$\mathcal{D}'$) is unramified at $v$.
Let $\Delta = \overline{\mathcal{D}}$ and $\Delta' =
\overline{\mathcal{D}}'$ be the corresponding residue algebras
(which are central division algebras over $k$ of the same dimension
$d^2 = \dim_{\mathcal{K}} \mathcal{D}= \dim_{\mathcal{K}}
\mathcal{D}'$). We claim that $\Delta$ and $\Delta'$ have the same
maximal subfields.
Indeed, let $\cO_{\mathcal{D}}$ and $\cO_{\mathcal{D}'}$ be the
valuation rings in $\mathcal{D}$ and $\mathcal{D}'$, respectively,
and let $\cO_{\mathcal{D}} \to \Delta$ and $\cO_{\mathcal{D}'} \to
\Delta'$ be the corresponding reduction maps (denoted $x \mapsto
\bar{x}$). Let $P$ be a maximal subfield of $\Delta$. Since $d$
divides $n$, hence is prime to $\mathrm{char}\: k$, the extension $P/k$
is separable, and therefore we can find $a \in \cO_{\mathcal{D}}$
such that $P = k(\bar{a})$. Set $F = K(a)$. We have
$$
d \geqslant [F : K] \geqslant [P : k] = d,
$$
which implies that $F$ is a maximal subfield of $\mathcal{D}$. By
our assumption, $F$ admits a $\mathcal{K}$-embedding into
$\mathcal{D}'$, and we let $b$ denote the image of $a$ under this
embedding. Then the subfield $P' = k(\bar{b})$ of $\Delta'$ is
$k$-isomorphic to $P$ and is maximal as $\dim_k \Delta' = d^2$.
Conversely, any maximal subfield of $\Delta'$ is $k$-isomorphic to a
maximal subfield of $\Delta$. This argument shows that
\begin{equation}\label{E:nuV}
\nu_v(\gen(D) \cap {}_n\Br(K)) \subset \gen(\Delta) \cap {}_n\Br(k).
\end{equation}
Since the composition $\nu_v \circ \iota_k$ coincides with the
identity map on ${}_n\Br(k)$, we conclude that the restriction of
$\nu_v$ to $[D]^{-1} \cdot \Lambda \subset \iota_k({}_n\Br(k))$ is
injective. On the other hand, it follows from (\ref{E:nuV}) that
$$
\nu_v([D]^{-1} \cdot \Lambda) \subset [\Delta]^{-1} \cdot
(\gen(\Delta) \cap {}_n\Br(k)).
$$
This yields both of the required facts for $\Lambda$ and concludes the proof.
\end{proof}

\vskip2mm

Next, we would like to point out a minor modification of Theorem
\ref{T:4}, which, under somewhat stronger assumptions, allows one to actually
bound the size of $\gen(D)$ and not just that of the intersection
$\gen(D) \cap {}_n\Br(K)$. Again, let $D$ be a central division
algebra over $K = k(C)$ of dimension $\dim_K D = \ell^2$, and
assume that $[D]$ has exponent $n$ in $\Br(K)$ (of course, $n \vert
\ell$, and, moreover, $n$ and $\ell$ have the same prime factors -
cf. \cite[Proposition 4.5.13]{GiSz}). Clearly $\gen(D) \subset
{}_{\ell}\Br(K)$. Now, if, as before, $n$ (and hence $\ell$) is
prime to $\mathrm{char}\: k$, then the map $\rho_V$ on ${}_n\Br(K)$
constructed in the proof of Theorem \ref{T:Ram1} extends to an
analogous map $\tilde{\rho}_V$ on ${}_{\ell}\Br(K)$, and we still
have the estimate
$$
\vert \tilde{\rho}_V(\gen(D)) \vert \leqslant \varphi(n)^r.
$$
Repeating almost verbatim the rest of the proof of Theorem \ref{T:4}, we
obtain

\vskip1mm

\begin{thm}\label{T:4'}
With notations as above, assume that

\vskip2mm

\noindent $\bullet$ the set $C(k)$ of $k$-rational points is
infinite;

\vskip1mm

\noindent $\bullet$ $\vert
{}_{\ell}\Br(K)_{\mathrm{ur}}/\iota_k({}_{\ell}\Br(k)) \vert =: M <
\infty$.

\vskip2mm

\noindent Then

\vskip2mm

\noindent {\rm (1)} \parbox[t]{16cm}{if there exists $N < \infty$
such that $\vert \gen(\Delta) \vert \leqslant N$ for any central
division $k$-algebra $\Delta$ of degree dividing $\ell$, then for
any central division $K$-algebra $D$ of degree dividing $\ell$, we have
$$
\vert \gen(D) \vert \leqslant M \cdot N \cdot \varphi(n)^r,
$$
where $r = \mathrm{Ram}_V (D)$;}

\vskip2mm

\noindent {\rm (2)} \parbox[t]{16cm}{if $\gen(\Delta)$ is finite for
any central division $k$-algebra $\Delta$ of degree dividing $\ell$,
then $\gen(D)$ is finite for any central division $K$-algebra of degree dividing $\ell.$}
\end{thm}

\vskip2mm

One notable case where these results apply is $K = k(x)$, i.e. $C =
\mathbb{P}^1_k$. Strictly speaking, the case of a finite field $k$ is not
covered by Theorems \ref{T:4} and \ref{T:4'}, so let us consider it
separately. In this case, the unramified Brauer group $\Br(K)_V$
(where, as above, $V$ is the set of all geometric places of $K$), is
trivial (cf. \cite[Corollary 6.4.6]{GiSz} and 3.6), so we obtain
from Theorem \ref{T:Ram1} that for any central division $K$-algebra
$D$ of exponent $n$ we have
$$
\vert \gen(D) \vert \leqslant \varphi(n)^r, \ \ \text{where} \ \ r =
\vert \mathrm{Ram}_V(D) \vert.
$$
Now, let us assume that $k$ is infinite. It is well-known that for
any $n$ prime to $\mathrm{char} \: k$, we have
${}_n\Br(K)_{\mathrm{ur}} = \iota_k({}_n\Br(k))$ (cf.
\cite[Corollary 6.4.6]{GiSz}), i.e. one can take $M = 1$ in Theorems
\ref{T:4} and \ref{T:4'}. We then obtain the
following.
\begin{thm}\label{T:Rat}
Let $K = k(x)$, and let $n > 1$ be an integer prime to
$\mathrm{char} \: k$.

\vskip2mm

\noindent {\rm (1)} \parbox[t]{16cm}{If there exists $N < \infty$
such that $\vert \gen(\Delta) \cap {}_n\Br(k) \vert \leqslant N$ for
any central division $k$-algebra $\Delta$ of exponent $n$, then for
any central division $K$-algebra $D$ of exponent $n$, we have $$\vert
\gen(D) \cap {}_n\Br(K) \vert \leqslant N \cdot \varphi(n)^r$$ where
$r = \vert \mathrm{Ram}_V(D) \vert$;}

\vskip2mm

\noindent {\rm (2)} \parbox[t]{16cm}{If there exists $N < \infty$
such that $\vert \gen(\Delta) \vert \leqslant N$ for any central
division $k$-algebra $\Delta$ of degree dividing $n$, then for any
central division $K$-algebra $D$ of degree dividing $n$, we have
$$\vert \gen(D) \vert \leqslant N \cdot \varphi(m)^r,$$ where $m$ is the exponent
of $D$ and $r = \vert \mathrm{Ram}_V(D) \vert$;}

\vskip2mm

\noindent {\rm (3)} \parbox[t]{16cm}{If $\gen(\Delta) \cap
{}_n\Br(k)$ (resp., $\gen(\Delta)$) is \emph{finite} for any central
division $k$-algebra $\Delta$ of exponent $n$ (resp., of degree
dividing $n$), then $\gen(D) \cap {}_n\Br(K)$ (resp., $\gen(D)$) is
\emph{finite} for any central division $K$-algebra $D$ of exponent
$n$ (resp., of degree dividing $n$).}
\end{thm}

\vskip2mm

\noindent {\bf Remark 3.4.} In \cite{KMcK}, Krashen and MacKinnie
defined the genus $\gen'(D)$ of a central division
$K$-algebra $D$ as the collection of $[D'] \in \Br(K)$ having the
same finite-dimensional splitting fields as $D$ (clearly $\gen'(D)
\subset \gen(D)$). Their Theorem 2.2 provides estimates for $\vert
\gen'(D) \vert$ of a central division algebra $D$ over $K = k(x)$
of a prime exponent $p \neq \mathrm{char} \: k$ similar to those given in
Theorem \ref{T:Rat}

\addtocounter{thm}{1}

\vskip2mm

Of special interest is the question of when $\gen(D)$ reduces to a
single element. As we already noted in \S \ref{S:Intro}, this is
possible only if $D$ has exponent two and  is indeed the case if $D$
is a division algebra of exponent two over a global field (see
below). Although over general fields this property may fail even for
quaternion algebras \cite[\S 2]{GS}, the following theorem allows
one to expand the class of fields over which it does hold.
\begin{thm}\label{T:1}\label{T:Stab1}
{\rm (Stability Theorem)} Let $k$ be a field of characteristic $\neq
2$.

\vskip2mm

\noindent {\rm (1)} If $k$ satisfies the following property:

\vskip2mm

\noindent $(*)$ \parbox[t]{16cm}{If $D$ and $D'$ are central
division $k$-algebras of exponent 2 having the same maximal
subfields then $D \simeq D'$ $($in other words, for any $D$ of
exponent 2, $\vert \gen(D) \cap {}_2\Br(k) \vert = 1)$.}

\vskip2mm

\noindent Then the field of rational functions $k(x)$ also satisfies
$(*)$.

\vskip2mm

\noindent {\rm (2)} \parbox[t]{16cm}{If $\vert \gen(D) \vert = 1$
for any central division $k$-algebra $D$ of exponent 2, then the
same is true for any central division $k(x)$-algebra of exponent 2.}
\end{thm}

This follows from Theorem \ref{T:Rat}, (1) and (2), with $n = 2$ and
$N = 1$.
%
%
%
%
%
%
%
%
%
%

\vskip2mm

\addtocounter{thm}{1}

\noindent {\bf 3.6. On the Albert-Brauer-Hasse-Noether Theorem.} In
this subsection, we will review several consequences of (ABHN) that
will be needed in Corollary \ref{C:Stab1} below as well as in the
next section. Let $k$ be a global field, and let $V^k$ be the set of
all places of $k$ (including the archimedean ones if $k$ is a number
field). In this case, the set $V^k \setminus V_{\infty}^k$, where
$V_{\infty}^k$ is the set of archimedean valuations, satisfies
conditions (A) and (B), and the residue map $\rho_v$ can be defined
on all of $\Br(k_v)$ for any  $v \in V^k \setminus V_{\infty}^k$
because the residue field $\overline{k}_v$ is finite, hence prefect.
Now, since the absolute Galois group $\mathcal{G}^{(v)}$ of
$\overline{k}_v$ is isomorphic to $\widehat{\Z}$, one can view
$\rho_v$ as a map
$$
\rho_v \colon \Br(k_v) \longrightarrow
\mathrm{Hom}(\mathcal{G}^{(v)} , \Q/\Z) \simeq \Q/\Z.
$$
This is usually referred to as the {\it invariant} map and is well
known to be an isomorphism (cf. \cite[Ch. XII, \S3]{Se}). For $v$ archimedean, we
have $\Br(k_v) = \Z/2\Z$ if $k_v = \R$ and $\Br(k_v) = 0$ if $k_v =
\C$, and one defines the invariant map $\rho_v$ to be the
isomorphism $\Z/2\Z \simeq (1/2)\Z/\Z$ in the first case and to be
the trivial map in the second. Then (ABHN) asserts that the sequence
\begin{equation}\tag{ABHN}\label{E:ABHN}
0 \to \Br(k) \longrightarrow \bigoplus_{v \in V^k} \Br(k_v)
\stackrel{\Sigma}{\longrightarrow} \Q/\Z \to 0,
\end{equation}
where $\Sigma$ is the sum of the invariant maps, is exact (cf.
\cite[Ch. VII, 9.6]{ANT}, \cite[18.4]{Pierce}, and also
\cite[6.5]{GiSz} for the function field case).



It follows from (ABHN) that for a global function field $k$,
the unramified Brauer group $\Br(k)_V$ with respect to the set $V =
V^k$ of all (geometric) places of $k$ is trivial. Applying Theorem
\ref{T:Ram1}, we obtain that $\gen(D)$ is finite for any central
division $k$-algebra $D$ and reduces to one element if $D$ is of
exponent two\footnote{We note that over a global field $k$, any
division algebra $D$ of exponent two is necessarily a quaternion
algebra, hence $\gen(D) \subset {}_2\Br(k)$.} (at least if
$\mathrm{char} k \neq 2$, although the result remains true in
characteristic two as well in view of Remark 2.6(2) - we give a
direct argument below).

\vskip1mm

Let now $k$ be a number field. Fix a finite subset $S \subset V^k$
containing the set $V_{\infty}^k$ of archimedean valuations, and set
$V = V^k \setminus S$. It follows from (ABHN) that for any $n > 1$
we have
\begin{equation}\label{E:ubr}
{}_n\Br(k)_V = \mathrm{Ker}\left(\bigoplus_{v \in S} {}_n\Br(k_v)
\stackrel{\Sigma_S}{\longrightarrow} \frac{1}{n}\Z/\Z \right),
\end{equation}
where $\Sigma_S$ is the sum of the invariant maps for $v \in S$.
Clearly, ${}_n\Br(k)_V$ is finite for any $n$ and $S$ as above (and
even is a group of exponent $\mathrm{g.c.d.}(n , 2)$ for $S =
V_{\infty}^k$), so Theorem \ref{T:Ram1} implies that $\gen(D)$ is
finite for any central division $k$-algebra $D$.
Unfortunately, ${}_2\Br(k)_V$ is nontrivial even for $S =
V_{\infty}^k$ if $k$ has at least two real places.
So, the argument used in the function field case to show that $\vert
\gen(D) \vert = 1$ for any central division $k$-algebra $D$ of
exponent two does not apply directly and needs to be modified to
also take into account the archimedean places. For this, we observe
that (ABHN) yields an embedding
$$
0 \to {}_2\Br(k) \longrightarrow \bigoplus_{v \in V^k} {}_2\Br(k_v),
$$
and that ${}_2\Br(k_v) = \Z/2\Z$ unless $k_v = \C$, in which case
$\Br(k_v) = 0$. It follows that any $[D] \in {}_2\Br(k)$,
represented by a quaternion algebra $D$, is completely determined by the set
$\mathscr{R}(D)$ of those $v \in V^k$ for which the algebra $D
\otimes_k k_v$ is nontrivial (these are sometimes referred to as the ``generalized ramification places").
Furthermore, for $d \in k^{\times} \setminus {k^{\times}}^2$, we
have the following well-known criterion:
\begin{equation}\label{E:emb2}
\ell = k(\sqrt{d}) \ \ \text{embeds \ into} \ \ D \ \ \
\Leftrightarrow \ \ \ d \notin {k_v^{\times}}^2 \ \ \text{for all} \
\ v \in \mathscr{R}(D).
\end{equation}
(cf. \cite[\S18.4, Corollary b]{Pierce}). Thus, to prove that $\vert \gen(D) \vert = 1$
for any quaternion division algebra $D$ over $k$, it is enough to
show that if two such algebras $D_1$ and $D_2$ have the same
quadratic subfields, then $\mathscr{R}(D_1) = \mathscr{R}(D_2)$.
This follows easily from (\ref{E:emb2}) and the weak approximation
theorem. Indeed, if, for example, there is a $v_0 \in \mathscr{R}(D_1)
\setminus \mathscr{R}(D_2)$, then using the openness of
${k_v^{\times}}^2 \subset k_v^{\times}$ and weak approximation, one
can find $d \in k^{\times} \setminus {k^{\times}}^2$ such that
$$
d \in {k_{v_0}^{\times}}^2 \ \ \text{but} \ \ d \notin
{k_v^{\times}}^2 \ \ \text{for all} \ \ v \in \mathscr{R}(D_2).
$$
Then according to (\ref{E:emb2}), the quadratic extension $\ell =
k(\sqrt{d})$ embeds into $D_2$ but not into $D_1$, a contradiction.

To make this more concrete, let us consider

\vskip2mm

\noindent {\bf Example 3.7.} Take the following two quaternion division algebras over $\Q$:
$$D_1 = \left( \frac{-1 , 3}{\Q}
\right) \ \ \text{and} \ \ D_2 = \left( \frac{-1 , 7}{\Q} \right).$$
Then $\mathscr{R}(D_1) = \{2 , 3 \}$ and $\mathscr{R}(D_2) = \{2 ,
7\}$. Clearly, $10 \in {\Q_3^{\times}}^2$ while $10 \notin
{\Q_2^{\times}}^2 , {\Q_7^{\times}}^2$.  So, by (\ref{E:emb2}), the field $\ell = \Q(\sqrt{10})$
embeds into $D_2$ but not into $D_1$. Thus, $D_1$ and $D_2$ are
distinguished by their quadratic subfields.

\vskip2mm

Notice that since the
subgroup of squares is open in $k_v$ for any valuation of any global
field of characteristic $\neq 2$, the argument given above works
for any such field. To extend it to characteristic two, one needs instead to use the fact that
the subgroup $\wp(k_v) \subset k_v$ is open, where $\wp(x) = x^2 - x$, and replace
(\ref{E:emb2}) with the observation that for $a \in k \setminus \wp(k)$,
$$
\ell = k(\wp^{-1}(a)) \ \ \text{embeds \  into} \ \ D \ \ \
\Leftrightarrow \ \ \ a \notin \wp(k_v) \ \ \text{for all} \ \ v \in
\mathscr{R}(D).
$$

\vskip2mm

Now, since $\vert \gen(D) \vert = 1$ for any central
division algebra $D$ of exponent two over a global field $k$, Theorem \ref{T:Stab1} yields the following.


\addtocounter{thm}{1}

\begin{cor}\label{C:Stab1}
Let $k$ be a field of characteristic $\neq 2$ which is either a
global field or a finite field, and let $K = k(x_1, \ldots , x_r)$
be a finitely generated purely transcendental extension of $k$. Then
$\vert \gen(D) \vert = 1$ for
any central division $K$-algebra $D$ of exponent 2.
\end{cor}

\vskip2mm

\noindent {\bf Remark 3.9.} As we already mentioned, Theorem 3 of
\cite{CRR} asserts that if $K$ is a finitely generated field, then
for any central division $K$-algebra $D$ of degree $n$ prime to
$\mathrm{char} \: K$, the genus $\gen(D)$ is finite. At the same
time, generalizing the construction described in \cite[\S 2]{GS},
one can give examples of quaternion division algebras over
infinitely generated fields with infinite genus (cf. \cite{Meyer}).
So, we would like to point out that that Theorems \ref{T:4} and
\ref{T:4'} can be used to construct examples of division algebras
over the function fields of curves with infinitely generated fields
of constants having finite genus. Some results of this nature are
contained in \cite[4.6 and 4.8]{RR}. We will not go into details
about this here, but would only like to point out that, as follows
from the standard exact sequence for the Brauer group of an
absolutely irreducible smooth projective curve $C$ over a perfect
field $k$ (cf. \cite{L} or \cite[(9.25) on p. 27]{GMS}), the
requirement $\vert {}_n\Br(K)_{\mathrm{ur}}/\iota_k({}_n\Br(k))
\vert < \infty$ is satisfied, for example, if $k$ is of type (F) as
defined by Serre (\cite{Serre}, Chap. III, Sect. 4.2). Furthermore,
the requirement that $C(k)$ is infinite can often be replaced by the
much weaker requirement that $\bigcup_{\ell} C(\ell)$, where $\ell$
runs through a family of finite extensions of $k$ of degree prime to
$n$, is infinite.

\vskip4mm

\section{An example: Function field of a split elliptic curve}\label{S:Ell}

According to Theorem \ref{T:Ram1}, to prove the finiteness of
$\gen(D)$ for any central division algebra $D$ of degree $n$ over a
field $K$ (provided that $n$ is prime to $\mathrm{char} \: K$), it
is enough to find a set $V$ of discrete valuations of $K$ that
satisfies conditions (A)-(C) and for which the unramified Brauer
group ${}_n\Br(K)_V$ is finite. As we already mentioned in \S
\ref{S:Intro}, in \cite{CRR} we sketched a proof of the finiteness
of ${}_n\Br(K)_V$ for a suitable set $V$ of discrete valuations of a
given finitely generated field $K$ that relies on results of Deligne
and Gabber in \'etale cohomology. While this proof has the important
advantage of giving a lot of flexibility in the choice of $V$, it
does not furnish an estimation of the size of ${}_n\Br(K)_V$ for
{\it any} $V$. In this section, we will work out an explicit
estimation of $\vert {}_2\Br(K)_V \vert$ for a suitable set $V$ of
discrete valuations of the field $K$ of rational functions on a
split elliptic curve defined over a number field (Theorem
\ref{T:Elliptic1}); this yields an estimation of $\vert \gen(D)
\vert$ for a quaternion algebra $D$ over such $K$ (Corollary
\ref{C:Elliptic1}). The method developed in this section can in fact
be generalized to arbitrary curves, leading to a more direct proof
of the finiteness of ${}_n\Br(K)_V$ for a suitable explicitly
defined set $V$ of discrete valuation of an arbitrary finitely
generated field $K$; details will be published elsewhere. To
simplify notations, in this section, given a field $F$, the
quaternion algebra $\displaystyle \left( \frac{\alpha \: , \:
\beta}{F} \right)$ corresponding to a pair $\alpha , \beta$ will be
denoted by $(\alpha , \beta)_F$.



\vskip2mm


Let $k$ be a number field, and let $E$ be an elliptic curve over $k$
given by a Weierstrass equation
\begin{equation}\label{E:Elliptic0}
y^2 = f(x), \ \ \text{where} \ \ f(x) = x^3 + \alpha x^2 + \beta x +
\gamma.
\end{equation}
Denote by $\delta \neq 0$ the discriminant of $f$. We will assume in
this section that $E$ {\it splits} over $k$, i.e. $f$ has three
(distinct) roots in $k$. Let
$$
K := k(E) = k(x , y)
$$
be the function field of $E$. For $s \in k^{\times}$, we let $V^k(s)$ denote the finite set
$\{ v \in V^k \setminus V_{\infty}^k \: \vert \: v(s) \neq 0 \}.$ Let us fix
a finite set of valuations $S
\subset V^k$ containing $V_{\infty}^k \cup V^k(2) \cup V^k(\delta)$,
as
well as all those nonarchimedean $v \in V^k$ for which at least one of
$\alpha, \beta, \gamma$ has a negative value. For a nonarchimedean $v
\in V^k$, let $\tilde{v}$ denote its extension to $F := k(y)$ given
by
\begin{equation}\label{E:Elliptic1}
\tilde{v}(p(y)) = \min_{a_i \neq 0} v(a_i) \ \ \text{for} \ \ p(y) =
a_n y^n + \cdots + a_0 \in k[y], \ \ p \neq 0
\end{equation}
(cf. \cite[Ch. VI, \S 10]{Bour}). Clearly, $K$ is a cubic extension
of $F$, and, as we will show in Lemma \ref{L:Elliptic3-2} below, for
$v \in V^k \setminus S$, the valuation $\tilde{v}$ has a {\it
unique} extension to $K$, which we will denote by $w = w(v)$.
We now
introduce the following set of discrete valuations of $K$:
$$
V = V_0 \cup V_1,
$$
where $V_0$ is the set of all {\it geometric} places of $K$ (i.e.,
those discrete valuations that are trivial on $k$), and $V_1$
consists of the valuations $w(v)$ for all $v \in V^k \setminus S$.
It is easy to see that $V$ satisfies conditions (A), (B) and (C) of
\S \ref{S:Ram}. The main result of this section is the following.
\begin{thm}\label{T:Elliptic1}
For \emph{any} finite set $S$ as above, the unramified Brauer group ${}_2\Br(K)_V$ is finite of order
dividing
$$
2^{\vert S \vert - t} \cdot \vert {}_2 \mathrm{Cl}_S(k) \vert^2
\cdot \vert U_S(k)/U_S(k)^2 \vert^2,
$$
where $t = c + 1$ and $c$ is the number of complex places of
$k$, and $\mathrm{Cl}_S(k)$ and $U_S(k)$ are the class group and the
group of units of the ring of $S$-integers $\mathcal{O}_k(S)$,
respectively.
\end{thm}

Our proof will make use of the following description of the
geometric Brauer group ${}_2\Br(K)_{V_0}$ in the split case, which
is valid for any field $k$ of characteristic $\neq 2, 3$.
\begin{thm}\label{T:Elliptic2}
{\rm (\cite[Theorem 3.6]{CGu})} Assume that the elliptic curve $E$
given by $(\ref{E:Elliptic0})$ splits over $k$, i.e. $$f(x)=(x -
a)(x - b)(x - c) \ \  \text{with} \ \  a, b, c \in k.$$ Then
$$
{}_2\Br(K)_{V_0} = {}_2\Br(k) \oplus I,
$$
where ${}_2\Br(k)$ is identified with a subgroup of ${}_2\Br(K)$ via
the canonical map $\Br(k) \to \Br(K)$, and $I \subset
{}_2\Br(K)_{V_0}$ is a subgroup such that every element of $I$ is
represented by a bi-quaternion algebra of the form
$$
(r , x - b)_K \otimes_K (s , x - c)_K
$$
for some $r , s \in k^{\times}$.
\end{thm}

\vskip1mm

Our argument for Theorem \ref{T:Elliptic1} will require us to consider separately
the ramification properties at places in
$V_1$ of the constant and bi-quaternionic parts of elements of
${}_2\Br(K)_{V_0}$. This analysis will be based on properties of the
{\it corestriction map} (cf. \cite[Ch. 8]{Salt}, \cite{Tign}). We
recall that given a finite separable field extension $K/F$, there is
a group homomorphism $\Cor_{K/F} \colon \Br(K) \to \Br(F)$ with the
following properties (cf. \cite{Tign}, Theorems 2.5 and 3.2):

\vskip2mm

(a) \parbox[t]{11cm}{the composition  $$\Br(F) \longrightarrow
\Br(K) \stackrel{\Cor_{K/F}}{\longrightarrow} \Br(F)$$ coincides
with multiplication by $[K : F]$;}

\vskip2mm

(b) \parbox[t]{13cm}{({\it projection formula}) if $\mathrm{char}\:
F \neq 2$, then for any $r \in F^{\times}$, $s \in K^{\times},$ we
have
$$
\Cor_{K/F}([(r , s)_{K}]) = [(r , N_{K/F}(s))_{F}].
$$}

\vskip3mm

(We note that the projection formula is valid not only for
quaternion algebras but for the symbol algebras of any degree $n$
provided that $n$ is prime to $\mathrm{char}\: K$ and $K$ contains a
primitive $n$th root of unity.) We will also need the following
statement, which easily follows from results proved in \cite{Salt}.
\begin{lemma}\label{L:Elliptic2}
Let $K/F$ be a finite separable field extension, $v$  a discrete
valuation of $F$, and $w$  an extension of $v$ to $K$. Assume that
\begin{equation}\label{E:Elliptic3-1}
[K : F] = [\overline{K}_w : \overline{F}_v]
\end{equation}
(then the extension $w$ is automatically \emph{unique}) and the
extension of the residue fields $\overline{K}_w/\overline{F}_v$ is
separable. If $n$ is prime to the characteristic of the residue
field $\overline{F}_v$ and $[A] \in {}_n\Br(K)$ is unramified at
$w,$ then $\Cor_{K/F}([A])$ is unramified at $v$.
\end{lemma}
\begin{proof}
First, (\ref{E:Elliptic3-1}) implies that the extension $w$ is
unique, and therefore  the valuation ring $\mathcal{O}_{K, w}$ is
the integral closure in $K$ of the valuation ring $\mathcal{O}_{F,
v}$. As $K/F$ is separable, it follows that $\mathcal{O}_{K, w}$ is
a free $\mathcal{O}_{F, v}$-module of rank $d = [K : F]$.
Furthermore, since the extension of residue fields
$\overline{K}_w/\overline{F}_v$ is also separable of degree $d$ (in
particular, $w \vert v$ is unramified), by \cite[Corollary
2.17]{Salt}, $\mathcal{O}_{K, w}$ is a separable $\mathcal{O}_{F,
v}$-algebra. The fact that $[A] \in {}_n\Br(K)$ is unramified at $w$
implies that  $A = \mathcal{A} \otimes_{\mathcal{O}_{K, w}} K$ for
some Azumaya $\mathcal{O}_{K, w}$-algebra $\mathcal{A}$ (cf.
\cite[Theorem 10.3]{Salt}). Then by \cite[Theorem 8.1, (a)]{Salt},
the corestriction (defined in {\it loc. cit.}) $\mathcal{B} :=
\Cor_{\mathcal{O}_{K, w}/\mathcal{O}_{F , v}}(\mathcal{A})$ is an
Azumaya $\mathcal{O}_{F, v}$-algebra, and by \cite[Theorem 8.1,
(d)]{Salt}, $\mathcal{B} \otimes_{\mathcal{O}_{F, v}} F$ represents
$\Cor_{K/F}([A])$, implying that the latter is unramified at $v$
(cf. \cite[Theorem 10.3]{Salt}), as required.
\end{proof}

\vskip4mm

For the rest of the section, we return to the notations $K = k(E)$
and $F = k(y)$. We will need the following two lemmas, the first of
which contains a simple computation and the second describes some
properties of the valuations $w \in V_1$ needed in the proof of Theorem
\ref{T:Elliptic1}.
\begin{lemma}\label{L:Elliptic1}
$(i)$ $N_{K/F}(x - a) = N_{K/F}(x - b) = N_{K/F}(x - c) = y^2.$

\noindent $(ii)$ Let $r \in k^{\times}$ and  $t \in \{a, b, c\}$,
then for the quaternion algebra $(r , x - t)_K$ we have
$$\Cor_{K/F}([(r , x - t)_K]) = 0.$$
\end{lemma}
\begin{proof}
$(i)$: We will only prove the claim for $x - a$ as all other cases are treated analogously.
It is easy to see that the minimal and
characteristic polynomials for $x - a$ over $F$ coincide with
$$
g(t) = t(t + (a - b))(t + (a - c)) - y^2.
$$
So, $N_{K/F}(x - a) = (-1)^3 \cdot (-y^2) = y^2$.

\vskip3mm

$(ii)$: Using $(i)$ and the projection formula, we obtain
$$
\Cor_{K/F}([(r , x - t)_K]) = \left[ (r , N_{K/F}(x - t))_F \right]
= \left[ (r , y^2)_F \right] = 0.
$$
\end{proof}

\begin{lemma}\label{L:Elliptic3-2}
Let $v \in V^k \setminus S$,  and let $\tilde{v}$ be the extension
of $v$ to $F = k(y)$ given by $(\ref{E:Elliptic1})$.


\vskip2mm

$(i)$
\parbox[t]{16cm}{For any extension $w$ of $\tilde{v}$ to $K$, the
residue field extension $\overline{K}_w/\overline{F}_v$ is a
separable cubic extension. Consequently, $\tilde{v}$ has
a~\emph{unique} extension $($to be denoted $w = w(v))$, which is
automatically unramified.}

\vskip2mm

$(ii)$ \parbox[t]{16cm}{For $w = w(v)$, the elements $x - a$, $x -
b$, $x - c$ are units with respect to $w$, and their images
$\overline{x - a}$, $\overline{x - b}$ and $\overline{x - c}$ in
$\overline{K}_w$ represent distinct nontrivial cosets in $
\overline{K}_w^{\times}/{ \overline{K}_w^{\times}}^2$.}
\end{lemma}
\begin{proof}
$(i)$: The relation
$$
x^3 + \alpha x^2 +  \beta x + (\gamma - y^2) = 0
$$
implies that $w(x) \geqslant 0$. Reducing, we obtain  $\phi(\bar{x})
= 0$ for the  polynomial
$$
\phi(t) = t^3 + \bar{\alpha} t^2 + \bar{\beta} t + (\bar{\gamma} -
\bar{y}^2)
$$
over $\overline{F}_{\tilde{v}} = \overline{k}_v(\bar{y})$, the field
of rational functions over the residue field $\overline{k}_v$ of
$k$. By considering the degree with respect to $\bar{y}$ one finds
that $\phi(t)$ has no roots in, hence is irreducible over,
$\overline{k}_v(\bar{y})$. Furthermore, for the discriminant
$\delta(\bar{y})$ of $\phi$ we have $\delta(0) = \bar{\delta}$,
where $\delta$ is the discriminant of $f$. So, it follows from our
choice of $S$ that $\delta(\bar{y}) \neq 0$, making $\phi$
separable. Thus, $\overline{k}_v(\bar{x} , \bar{y})$ is a separable
cubic extension of $\overline{k}_v(\bar{y})$. By degree
considerations, $\overline{K}_w = \overline{k}_v(\bar{x} ,
\bar{y})$, and all of our assertions follow.

\vskip2mm

$(ii)$: Since $w(y^2) = 0$ and $w(x) \geqslant 0$,
$$y^2 = (x - a)(x - b)(x - c)$$ yields that $x - a$, $x - b$ and $x
- c$ are all $w$-units. Furthermore, our assumption that
$\bar{\delta} \neq 0$ means that the residues $\bar{a}, \bar{b},
\bar{c}$, hence the residues $\bar{x} - \bar{a}$, $\bar{x} -
\bar{b}$ and $\bar{x} - \bar{c}$, are pairwise distinct. First, let
us show that no $d \in \{\bar{x} - \bar{a}, \bar{x} - \bar{b} ,
\bar{x} - \bar{c} \}$ can be a square in $\overline{K}_w$. Indeed,
if, for example, $d = \bar{x} - \bar{a} \in
{\overline{K}_w^{\times}}^2$ then since $d \notin {
\overline{k}_v(\bar{x})^{\times}}^2$ and $[\overline{K}_w :
\overline{k}_v(\bar{x})] \leqslant 2$, we obtain that
$$
\overline{K}_w = \overline{k}_v(\bar{x})(\bar{y}) =
\overline{k}_v(\bar{x})(\sqrt{d}).
$$
Consequently,
$$
\frac{\bar{y}^2}{d} = (\bar{x} - \bar{b})(\bar{x} - \bar{c}) \in
{\overline{k}_v(\bar{x})^{\times}}^2,
$$
which is impossible as $\bar{b} \neq \bar{c}$. Thus, $\bar{x} -
\bar{a}$ is not a square in $\overline{K}_w$. Furthermore, if, for
example, $\bar{x} - \bar{b}$ and $\bar{x} - \bar{c}$ would represent
the same coset modulo ${\overline{K}_w^{\times}}^2$ then $\bar{x} -
\bar{a}$ would be a square in $\overline{K}_w$, which is not the
case.
\end{proof}

\vskip2mm \noindent {\bf Remark 4.6.} Using the uniqueness statement in
\cite[Ch. VI, \S 10, Proposition~2]{Bour}, one easily proves that
the restriction of $w$ to $k(x)$ is given by
$$
w(q(x)) = \min_{b_j \neq 0} v(b_j) \ \ \text{for} \ \ q(x) = b_mx^m
+ \cdots + b_0 \in k[x] \setminus \{ 0 \}.
$$

\vskip2mm

\addtocounter{thm}{1}

We are now in a position to determine when bi-quaternion algebras of
the form described in Theorem \ref{T:Elliptic2} are unramified at
places $w \in V_1$.
\begin{prop}\label{P:Elliptic-ramif}
Let $v \in V^k \setminus S$, and let $w = w(v)$ be the corresponding
valuation of $K$ (see Lemma \ref{L:Elliptic3-2}). If
$$
\Delta = (r , x-b)_K \otimes_K (s , x - c)_K
$$
is unramified at $w$ then
$$
v(r) \; , v(s) \; \equiv 0(\mathrm{mod}\: 2).
$$
\end{prop}
\begin{proof}
We will need the following well-known description of the values of
the residue map
$$
\rho_w \colon {}_2\Br(K) \longrightarrow
\mathrm{Hom}(\mathcal{G}^{(w)} , \Q/\Z)
$$
(recall that by our construction $\mathrm{char}\: \overline{K}_w
\neq 2$) on quaternion algebras. For $h \in K^{\times}$ such that
$w(h) = 0$, we define $\kappa_h \colon \mathcal{G}^{(w)} \to \Z/2\Z$
by
$$
\kappa_h(\sigma) = \left\{ \begin{array}{ccl} 0(\mathrm{mod}\: 2) &
\text{if} & \sigma(\sqrt{\bar{h}}) = \sqrt{\bar{h}}, \\
1(\mathrm{mod}\: 2) & \text{if} & \sigma(\sqrt{\bar{h}}) =
-\sqrt{\bar{h}}, \end{array} \right. \hskip4mm \text{for} \ \ \sigma
\in \mathcal{G}^{(w)},
$$
where $\bar{h} \in \overline{K}_w^{\times}$ is the residue of $h$.
Then given $g , h \in K^{\times}$ with  $w(h) = 0$, we have
$$
\rho_w([(g , h)_K])(\sigma) = \frac{w(g)\kappa_{h}(\sigma)}{2}
(\mathrm{mod}\: \Z).
$$
Applying this to $\Delta$ as in the statement of the proposition, we
obtain
\begin{equation}\label{E:Elliptic20}
\rho_w([\Delta])(\sigma) = \frac{v(r) \kappa_{(x - b)}(\sigma) +
v(s)\kappa_{(x - c)}(\sigma)}{2} (\mathrm{mod}\: \Z).
\end{equation}
Since by Lemma \ref{L:Elliptic3-2}$(ii)$, the elements $\overline{x
- b}$ and $\overline{x - c}$ represent different nontrivial cosets
in $\overline{K}_w^{\times} / {\overline{K}_w^{\times}}^2$, the map
$$
\mathcal{G}^{(w)} \longrightarrow \Z/2\Z \times \Z/2\Z, \ \ \ \sigma
\mapsto (\kappa_{(x - b)}(\sigma) , \kappa_{(x - a)}(\sigma))
$$
is surjective. Using this in conjunction with the fact that
$\rho_w([\Delta])$ given by (\ref{E:Elliptic20}) is actually
trivial, we easily obtain our claim.
\end{proof}

\vskip2mm

\vskip1mm

\noindent {\it Proof of Theorem \ref{T:Elliptic1}.} Let $[D] \in
{}_2\Br(K)_V$. According to Theorem \ref{T:Elliptic2}, we can write
\begin{equation}\label{E:Elliptic4-1}
[D] = [\Delta' \otimes_K \Delta''],
\end{equation}
where $\Delta' = \Delta_0 \otimes_k K$ for some central division
algebra $\Delta_0$ over $k$ such that $[\Delta_0] \in {}_2\Br(k)$,
and
$$
\Delta'' = (r , x - b)_K \otimes_K (s , x - c)_K
$$
for some $r , s \in k^{\times}$. The next lemma describes some
restrictions on $\Delta_0$ and $\Delta''$, which will enable us to
limit the number of possibilities for these algebras and eventually
prove the theorem.
\begin{lemma}\label{L:Elliptic3}
$(i)$ $\Delta_0$ is unramified at all $v \in V^k \setminus S$.

\noindent $(ii)$ $\Delta''$ is unramified at all $w \in V_1$.
\end{lemma}
\begin{proof}
$(i)$: Fix $v \in V^k \setminus S$, and let $w = w(v)$. Since $[D]$
is unramified at $w$, it follows from Lemma \ref{L:Elliptic2} in
conjunction with Lemma \ref{L:Elliptic3-2}$(i)$ that
$\Cor_{K/F}([D])$ is unramified at $\tilde{v}$. On the other hand,
by Lemma \ref{L:Elliptic1}$(ii)$ we have that
$\Cor_{K/F}([\Delta'']) = 0$, and therefore
$$
\Cor_{K/F}([D]) = \Cor_{K/F}([\Delta']) = 3 \cdot  [\Delta_0
\otimes_k F] = [\Delta_0 \otimes_k F]
$$
as $\Delta' = (\Delta_0 \otimes_k F) \otimes_F K$. Thus, $\Delta_0
\otimes_k F$ is unramified at $\tilde{v}$, which implies  that
$\Delta_0$ is unramified at $v$, as required. Indeed, since
$[\Delta_0] \in {}_2\Br(k)$ and $k$ is a number field, we can take
$\Delta_0$ to be a quaternion algebra. It follows from the
description of the residue map we have already used in the proof of
Proposition \ref{P:Elliptic-ramif} that if $\Delta_0$ is ramified at
$v$, then it can be written in the form $\Delta_0 = (r , s)_k$, with
$r , s \in k^{\times}$, where $v(r) = 0$ and $\bar{r} \notin
{\overline{k}_v^{\times}}^2$ and $v(s) = 1$. Then $\Delta_0
\otimes_k F = (r , s)_F$. Furthermore, since
$\overline{F}_{\tilde{v}} = \overline{k}_v(\bar{y})$, we see that
$\bar{r} \notin {\overline{F}_{\tilde{v}}^{\times}}^2$. As
$\tilde{v}(s) = v(s) = 1$, we conclude that $\Delta_0 \otimes_k F$
is ramified at $\tilde{v}$, a contradiction.

\vskip1mm

$(ii)$: Since $\Delta_0$ is unramified at all $v \in V^K \setminus
S$, it is easy to see (e.g. using Azumaya algebras) that $\Delta'$
is unramified at all $w \in V_1$. So, $\Delta'' = (r , x - b)_K
\otimes_K (s , x - c)_K$ is unramified at all $w \in V_1$.
\end{proof}

\vskip1mm

Thus, the class $[\Delta_0]$ belongs to the unramified Brauer group
${}_2\Br(k)_{V^k \setminus S}$. So, the following immediate
consequence of (ABHN) bounds the number of possibilities for
$[\Delta_0]$.
\begin{lemma}\label{L:Elliptic15}
Let $S \subset V^k$ be a finite subset containing $V_{\infty}^k$ and
at least one non-complex place. Then
$$
\vert {}_2\Br(k)_{V^k \setminus S} \vert = 2^{\vert S \vert - t},
$$
where $t = c + 1$ and $c$ is the number of complex places.
\end{lemma}
\begin{proof}
According to (\ref{E:ubr}), we have
$$
{}_2\Br(K)_{V^k \setminus S} \simeq \mathrm{Ker}\left( \bigoplus_{v
\in S}\: {}_2\Br(k_v) \stackrel{\Sigma_S}{\longrightarrow}
\frac{1}{2}\Z/\Z \right),
$$
where $\Sigma_S$ is the sum of the invariant maps for $v \in S$.
Since ${}_2\Br(k_v) \simeq (1/2)\Z/\Z$ for every non-complex $v$ and
$\Sigma_S$ is surjective as $S$ contains a non-complex place, our
assertion follows.
\end{proof}

\vskip2mm

To bound the number of possibilities for $[\Delta'']$, we need the
following well-known statement.
\begin{lemma}\label{L:Elliptic10}
Let $k$ be a number field and $S \subset V^k$ be a finite
subset containing $V_{\infty}^k$. Set
$$
\tilde{\Gamma} = \{ x \in k^{\times} \: \vert \: v(x) \equiv
0(\mathrm{mod}\: 2) \ \ \text{for all} \ \ v \in V^k \setminus S \}.
$$
If $\nu_2 \colon k^{\times} \to k^{\times} / {k^{\times}}^2$ is the
canonical homomorphism, then the image  $\Gamma = \nu_2(\tilde{\Gamma})$ is
finite of order $$\vert \Gamma \vert = \vert
{}_2\mathrm{Cl}_S(k) \vert \cdot \vert U_S(k)/U_S(k)^2 \vert,$$ where $\mathrm{Cl}_S(k)$ and $U_S(k)$ are the class group and the
group of units of the ring of $S$-integers $\mathcal{O}_k(S)$, respectively.
\end{lemma}
\begin{proof}
According to \cite[Ch. 6, Theorem 1.4]{LaDG}, there is an exact sequence
$$
0 \to U_S(k)/U_S(k)^2 \longrightarrow \Gamma \longrightarrow
{}_2\mathrm{Cl}_S(k) \to 0,
$$
from which our claim follows.
\end{proof}

\vskip1mm

By Lemma \ref{L:Elliptic3}$(i)$, $[\Delta'']$ is unramified at $w =
w(v)$ for any $v \in V^k \setminus S$, so it follows from
Proposition~\ref{P:Elliptic-ramif} that $r , s \in \tilde{\Gamma}$.
Then the number of possibilities for $[\Delta'']$ does not exceed
$\vert \Gamma \vert^2$. So, the required estimation in Theorem
\ref{T:Elliptic1} is obtained by combining the estimations from
Lemmas \ref{L:Elliptic15} and \ref{L:Elliptic10}. \hfill $\Box$
%
%
%
%
%

\vskip3mm

\begin{cor}\label{C:Elliptic1}
In the notations of Theorem \ref{T:Elliptic1}, for any central quaternion division algebra $D$ over $K,$ we have
$$
\vert \gen(D) \vert \leqslant 2^{\vert S \vert - t} \cdot \vert {}_2
\mathrm{Cl}_S(k) \vert^2 \cdot \vert U_S(k)/U_S(k)^2 \vert^2.
$$
\end{cor}

\vskip5mm

\noindent {\bf Example 4.12.} Consider the elliptic curve $E$ over
$\Q$ given by $y^2 = x^3 - x$. We have $\delta = 4$, so $S =
\{\infty , 2 \}$. Furthermore,
$$
\vert S \vert - t = 1, \ \ \mathrm{Cl}_S(\Q) = 1 \ \ \text{and} \ \
U_S(\Q) = \{\pm 1\} \times \Z.
$$
So, by Corollary \ref{C:Elliptic1}
we have $\vert \gen(D) \vert \leqslant 2 \cdot 4^2 = 32$.

\vskip5mm

\section{Generalizations}\label{S:Gen}


One can generalize the notion of genus from division algebras to
algebraic groups using maximal tori in place of maximal subfields.
More precisely, given an absolutely almost simple (simply connected
or adjoint) algebraic $K$-group $G$, we define its genus as the set
of $K$-isomorphism classes of $K$-forms $G'$ of $G$ that have the
same isomorphism (or isogeny) classes of maximal $K$-tori as $G$ (a
different approach to the definition of genus is described in Remark
5.6 below).

\vskip1.5mm

\noindent {\bf Remark 5.1.} We note that for $G = \mathrm{SL}_{1 ,
D}$, where $D$ is a finite-dimensional central division $K$-algebra,
only maximal {\it separable} subfields of $D$ give rise to maximal
$K$-tori of $G$. So, to make the definition of $\gen(D)$ consistent
with this definition of the genus of $G$, one should probably give
the former in terms of maximal separable subfields rather than in
terms of all subfields. In this paper, however, we consider only
division algebras whose degree is prime to $\mathrm{char} \: K$, for
which this issue does not arise, so we opted to use the simplest
possible definition of $\gen(D)$.

\vskip1.5mm

In view of the finiteness theorem for $\gen(D)$ \cite{CRR}, it seems
natural to propose the following.

\vskip1.2mm

\noindent {\bf Conjecture 5.2.} {\it Let $G$ be an absolutely almost
simple simply connected algebraic group over a~finitely generated
field $K$ of characteristic zero (or of ``good" characteristic
relative to $G$). Then there exists a \emph{finite} collection $G_1,
\ldots , G_r$ of $K$-forms of $G$ such that if $H$ is a $K$-form of
$G$ having the same isomorphism classes of maximal $K$-tori as $G$,
then $H$ is $K$-isomorphic to one of the $G_i$'s.}

\vskip1.2mm

\addtocounter{thm}{2}

It was shown in \cite[Theorem 7.5]{PR1} that the conjecture is true
if $K$ is a number field. Furthermore, our previous results enable
us to prove this conjecture for inner forms of type
$\textsf{A}_{\ell}$ in the general case.

\begin{thm}\label{T:Al}
Let $G$ be an absolutely almost simple simply connected algebraic
group of inner type $\textsf{A}_{\ell}$ over a finitely generated
field $K$ whose characteristic is either zero or does not divide
$\ell + 1$. Then the above conjecture is true for $G$.
\end{thm}
\begin{proof}
We recall that $G = \mathrm{SL}_{m , D}$, where $D$ is a central
division $K$-algebra of degree $n$ with $mn = \ell + 1$ (cf.
\cite[2.3.1]{Pl-R}) It is well-known that any maximal $K$-torus $T$
of $G$ is of the form
$$
\mathrm{R}^{(1)}_{E/K}(\mathbb{G}_m) =
\mathrm{R}_{E/K}(\mathbb{G}_m) \cap G
$$
(where $\mathbb{G}_m$ is the 1-dimensional split torus and
$\mathrm{R}_{E/K}$ denotes the Weil functor of restriction of
scalars) for some maximal \'etale subalgebra $E$ of $A = M_m(D)$.
Let $G'$ be a $K$-form of $G$ having the same isomorphism classes of
maximal $K$-tori as $G$. To prove the theorem, it is enough to
establish the following two facts:

\vskip2mm

\noindent \ (I) \parbox[t]{16cm}{$G'$ is an inner form over $K$, and
consequently $G' = \mathrm{SL}_{m , D'}$, where $D'$ is a central
division $K$-algebra of degree $n$;}

\vskip1.3mm

\noindent (II) \parbox[t]{16cm}{for any discrete valuation $v$ of
$K$, write $D \otimes_K K_v = M_s(\mathcal{D})$ and $D' \otimes_K
K_v = M_{s'}(\mathcal{D}')$ with $\mathcal{D} , \mathcal{D}'$
division algebras; then $s = s'$, and $\mathcal{D}$ and
$\mathcal{D}'$ have the same isomorphism classes of maximal
subfields.}

\vskip2mm

\noindent Indeed, according to Theorem 8 in \cite{CRR}, there exists
a set $V$ of discrete valuations of $K$ that satisfies conditions
(A), (B) and (C) and for which the unramified Brauer group
${}_n\Br(K)_V$ is finite. Then as we pointed out in Remark 2.6(1),
the local genus $\gen_V(D)$ is finite. On the other hand, it follows
from (II) that $[D'] \in \gen_V(D)$, so the finiteness of the genus
of $G$ follows.

\vskip2mm

Regarding (II), we note that the issue here is that for $T_i =
\mathrm{R}^{(1)}_{E_i/K}(\mathbb{G}_m)$, $i = 1, 2$ where $E_1 ,
E_2$ are \'etale $K$-algebras, a $K$-defined isomorphism $T_1 \simeq
T_2$ of tori may not be induced by an isomorphism $E_1 \simeq E_2$
of $K$-algebras. So, the crucial observation is that the former will
in fact be (essentially) induced by the latter for {\it generic
tori}. We will now recall the relevant definitions that apply to an
arbitrary semi-simple $K$-group $G$, and then return to the group
$G$ as in Theorem \ref{T:Al}. Let $T$ be a maximal $K$-torus of $G$,
and denote by $\Phi = \Phi(G , T)$ the corresponding root system.
Furthermore, let $K_{T}$ be the minimal splitting field of $T$ and
$\Theta_{T} = \Ga(K_{T}/K)$ be its Galois group. Then the natural
action of $\Theta_{T}$ on the character group $X(T)$ defines a
homomorphism
$$
\theta_{T} \colon \Theta_{T} \longrightarrow \mathrm{Aut}(\Phi),
$$
and we say that $T$ is generic (over $K$) if $\mathrm{Im} \:
\theta_{T}$ contains the Weyl group $W(\Phi) = W(G , T)$. If $K$ is
a finitely generated field then generic tori always exist; moreover
we have the following.
\begin{prop}\label{P:generic}
Let $G$ be an absolutely simple algebraic group over a finitely
generated field $K$. Given a discrete valuation $v$ of $K$ and a
maximal $K_v$-torus $T_v$ of $G$, there exists a maximal $K$-torus
$T$ of $G$ which is generic over $K$ and is conjugate to $T_v$ by an
element of $G(K_v)$.
\end{prop}
If $K$ is of characteristic zero, this is proved in \cite[Corollary
3.2]{PR-Fields}; the argument in positive characteristic requires
only minimal changes.

\vskip2mm

{\it Proof of} (I). Let now $G$ be as in Theorem \ref{T:Al}. Since
it is an inner form over $K$, for any maximal $K$-torus $T$ of $G$
we have $\mathrm{Im} \: \theta_T \subset W(G , T)$ (cf. \cite[Lemma
4.1(b)]{PR1}). Assume now that $G'$ is an outer form. Using
Proposition \ref{P:generic}, pick a maximal generic $K$-torus $T'$
of $G'$. Since $G'$ is an outer form, we have $\mathrm{Im} \:
\theta_{T'} \not\subset W(\Phi(G' , T'))$, and therefore eventually
$\mathrm{Im} \: \theta_{T'} = \mathrm{Aut}(\Phi(G' , T'))$. It
follows from our previous remark that $T'$ cannot be $K$-isomorphic
to any maximal $K$-defined torus of $G$, a contradiction. Thus, $G'
= \mathrm{SL}_{m' , D'}$. Furthermore, the fact that $G$ and $G'$
have the same isomorphism classes of maximal $K$-tori, implies that
$$
m - 1 = \mathrm{rk}_K \: G = \mathrm{rk}_K \: G' = m' - 1,
$$
i.e. $m = m'$, and hence $D'$ has degree $n$.

\vskip2mm

{\it Proof of} (II). Let $A = M_m(D)$ and $A' = M_m(D')$. It is
enough to show that for a discrete valuation $v$ of $K$, the
algebras $\mathscr{A}_v = A \otimes_K K_v$ and $\mathscr{A}'_v = A'
\otimes_K K_v$ have the same isomorphism classes of maximal \'etale
subalgebras (cf. the Lemma 2.3 and Corollary 2.4 in \cite{RR}). Fix
a maximal \'etale $K_v$-subalgebra $\mathscr{E}$ of $\mathscr{A}_v$,
and let $\mathscr{T} =
\mathrm{R}^{(1)}_{\mathscr{E}/K_v}(\mathbb{G}_m)$ be the
corresponding maximal $K_v$-torus of $G$. Using Proposition
\ref{P:generic}, we can find a maximal $K$-torus $T$ of $G$ which is
generic over $K$ and which is conjugate to $\mathscr{T}$ by an
element of $G(K_v)$. We have $T =
\mathrm{R}^{(1)}_{E/K}(\mathbb{G}_m)$ for some maximal \'etale
$K$-subalgebra $E$ of $A$, and then $E \otimes_K K_v$ and
$\mathscr{E}$ are isomorphic as $K_v$-algebras. By our assumption,
there exists a $K$-isomorphism $\varphi \colon T \to T'$ onto a
maximal $K$-torus $T'$ of $G'$, where $T' =
\mathrm{R}_{E'/K}(\mathbb{G}_m)$. Since $T$ is generic and $G , G'$
are inner forms, it follows from Lemma 4.3 and Remark 4.4 in
\cite{PR1} that $\varphi$ extends to an isomorphism $\tilde{\varphi}
\colon G \to G'$ defined over an algebraic closure $\overline{K}$ of
$K$. Now, pick an isomorphism $A \otimes_K \overline{K} \to A'
\otimes_K \overline{K}$ of $\overline{K}$-algebras, and let
$\varphi_0 \colon G \to G'$ be the corresponding
$\overline{K}$-isomorphism of algebraic groups. Set $\varphi' =
\varphi$ if $\psi = \varphi^{-1}_0 \circ \tilde{\varphi} \in
\mathrm{Aut} \: G$ is inner, and define $\varphi'$ by $\varphi'(t) =
\varphi(t)^{-1}$ for $t \in T$ if $\psi$ is outer. Then in either
case, $\varphi' \colon T \to T'$ is a $K$-defined isomorphism of
tori that extends to a $\overline{K}$-isomorphism $\tilde{\varphi}'
\colon G \to G'$ induced by {\it some} isomorphism of
$\overline{K}$-algebras $\tau \colon A \otimes_K \overline{K} \to A'
\otimes_K \overline{K}$. Since $E$ (resp., $E'$) coincides with the
$K$-subalgebra of $A$ (resp., $A'$) generated by $T(K)$ (resp.,
$T'(K)$), and $\varphi'(T(K)) = T'(K)$, we conclude that $\tau$
yields a $K$-isomorphism between $E$ and $E'$. It follows that
$\mathscr{E}$ is isomorphic to the maximal \'etale $K_v$-subalgebra
$E' \otimes_K K_v$ of $\mathscr{A}'_v$. By symmetry, we see that
that $\mathscr{A}_v$ and $\mathscr{A}'_v$
have the same isomorphism classes of maximal \'etale
$K_v$-subalgebras, as required.
\end{proof}

\vskip2mm

\noindent {\bf Remark 5.5.} In the notations introduced in the proof
of (II) above, it follows from Lemma 4.3 and Remark 4.4 in
\cite{PR1} that if there exists a nontrivial $K$-defined {\it
isogeny} $T \to T'$ of generic tori in absolutely almost simple
simply connected groups of type $\textsf{A}_{\ell}$ then there also
exists a $K$-defined {\it isomorphism} $T \to T'$. So, the
finiteness result of Theorem \ref{T:Al} remains valid if one defines
the genus using isogeny classes in place of isomorphism classes. In
fact, it also remains valid if one defines the genus in terms of
isomorphism/isogeny classes of just maximal {\it generic} tori. On
the other hand, the argument used to prove Theorem 1 in \cite{PR1}
shows that if $G$ and $G'$ are absolutely almost simple algebraic
groups over a finitely generated field $K$ with the same isogeny
classes of maximal $K$-tori then either they have the same type, or
one of them is of type $\textsf{B}_{\ell}$ and the other of type
$\textsf{C}_{\ell}$ for some $\ell \geqslant 3$. In particular, if
$G$ is an absolutely almost simple simply connected algebraic group
of type different from $\textsf{B}_{\ell}$ $(\ell \geqslant 2)$ over
a finitely generated field $K$, then any absolutely almost simple
simply connected $K$-group having the same isogeny classes of
maximal $K$-tori as $G$ is necessarily a $K$-form of $G$.

\vskip2mm

\noindent {\bf Remark 5.6.} (Due to A.S.~Merkurjev) One can offer a
different (in a way, more functorial) definition of the genus of an
absolutely almost simple $K$-group $G$ as the set of $K$-isomorphism
classes of $K$-forms $G'$ of $G$ that have the same
isomorphism/isogeny classes of maximal tori not only over $K$ but
also over any field extension $F/K$. The well-known theorem of
Amitsur \cite{Ami} asserts that if $D$ and $D'$ are
finite-dimensional central division $K$-algebras such that every
field extension $F/K$ which splits $D$ also splits $D'$ then $[D']$
lies in the cyclic subgroup $\langle [D] \rangle$ of $\Br(K)$
generated by $[D]$; in particular, the genus of $G = \mathrm{SL}_{1
, D}$ would then be finite for any $D$ and would reduce to one
element for $D$ of exponent two. Furthermore, according to a result
of Izhboldin \cite{Izhb}, given nondegenerate quadratic forms $q$
and $q'$ of {\it odd} dimension $n$ over a~field $K$ of
characteristic $\neq 2$ the following condition

\vskip1mm

\noindent $(\star)$ {\it $q$ and $q'$ have the same Witt index over
any extension $F/K$,}

\vskip1mm

\noindent implies that $q$ and $q'$ are scalar multiples of each
other (this conclusion being false for even-dimensional forms). It
follows that for $G = \mathrm{Spin}_n(q)$ with $n$ odd the genus of
$G$ as defined in this remark reduces to a single element. We note
that the condition $(\star)$ is equivalent to the fact that the
motives of $q$ and $q'$ in the category of Chow motives are
isomorphic (Vishik \cite{Vish1}, and also Vishik \cite[Theorem
4.18]{Vish2}, Karpenko \cite{Karp}), so one can call the genus
defined above the {\it motivic genus}. It would be interesting to
investigate the motivic genus for other types of algebraic groups;
e.g. one can expect it  trivial for type $\textsf{C}_{\ell}$ (note
that it easily follows from properties of the Pfister forms that it
is always trivial for type $\textsf{G}_2$, at least in
characteristic not 2).

\vskip3mm

\section{A geometric connection}\label{S:Geom}

The recent interest in the problem of determining an
absolutely almost simple algebraic $K$-group by the isomorphism/isogeny
classes of its maximal $K$-tori (in particular, of determining a
finite-dimensional central division $K$-algebra by the isomorphism
classes of maximal subfields) was motivated at least in part by the
investigation of length-commensurable and isospectral locally
symmetric spaces (cf. \cite[\S 6]{PR-Gen}). In fact, the initial
question of whether two central quaternion division algebras over
$\Q(x)$ with the same quadratic subfields are necessarily
isomorphic appeared in print (apparently, for the first time)
in the paper \cite{PR1} on locally symmetric space (although
unofficially it may have been around for some time). The purpose of
this section is to briefly recall some aspects of this geometric
connection in order to provide a context for a general conjecture
about Zariski-dense subgroups.

For a (compact) Riemannian manifold $M$, we let $\mathcal{E}(M)$
denote the Laplace spectrum of $M$ (i.e, the spectrum of the
Beltrami-Laplace operator) and $L(M)$ the (weak) length spectrum of
$M$ (i.e., the collection of lengths of all closed geodesics in
$M$). It is a classical problem in differential geometry to
determine what one can say about two Riemannian manifolds $M_1$ and
$M_2$ given the fact that they are {\it isospectral}, i.e.
$\mathcal{E}(M_1) = \mathcal{E}(M_2)$. It turns out that for locally
symmetric spaces, isospectrality implies iso-length-spectrality
($L(M_1) = L(M_2)$), see \cite[Theorem 10.1]{PR1}. On the other
hand, neither iso-length-spectrality nor even isopectrality
typically implies that $M_1$ and $M_2$ are isometric, although it
was established in \cite{PR1} that for arithmetically defined
locally symmetric spaces of simple real algebraic groups of types
different from $\textsf{A}_{\ell}$, $\textsf{D}_{2\ell + 1}$ $(\ell
> 1)$ and $\textsf{E}_6$, the weaker condition of
length-commensurability ($\Q \cdot L(M_1) = \Q \cdot L(M_2)$)
already implies that $M_1$ and $M_2$ are {\it commensurable}, i.e.
have a common finite-sheeted cover. We refer the reader to
\cite{PR-Gen} for the history of the problem and an exposition of
the available results. We note only that the problem remains wide
open for non-arithmetically defined locally symmetric spaces, and
will now show that its analysis in the case of Riemann surfaces
leads to a variant of the problem about quaternion algebras with the
same maximal fields.

Let $\mathbb{H} = \{ x + iy \in \C \: \vert \: y > 0 \}$  be the
upper half-plane with the standard hyperbolic metric $$ds^2 =
y^{-2}(dx^2 + dy^2).$$ Any compact Riemann surface $M$ of genus $g
> 1$ can be written as a quotient $M = \mathbb{H}/\Gamma$ by a
(cocompact) discrete torsion-free subgroup $\Gamma \subset
PSL_2(\R)$. Let $\pi \colon SL_2(\R) \to PSL_2(\R)$ be the canonical
homomorphism, and set $\tilde{\Gamma} = \pi^{-1}(\Gamma)$. It is
well-known that closed geodesics in $M$ correspond to nontrivial
semi-simple elements in $\Gamma$; the precise nature
of this
correspondence is not important for us, as we only need information
about the length. One shows that if $c_{\gamma}$ is a closed
geodesic in $M$ corresponding to a nontrivial semi-simple element
$\gamma \in \Gamma$, then its length is given by the formula:
$$
\ell(c_{\gamma}) = \frac{2}{n_{\gamma}} \cdot \vert \log \vert
t_{\tilde{\gamma}} \vert \vert
$$
where $n_{\gamma}$ is an integer $\geqslant 1$ (winding number) and
$t_{\tilde{\gamma}}$ is an eigenvalue of an element $\tilde{\gamma}
\in \tilde{\Gamma}$ such that $\pi(\tilde{\gamma}) = \gamma$ (note
that since $\Gamma$ is torsion-free and discrete, $\tilde{\gamma}$
is automatically hyperbolic, i.e. $t_{\tilde{\gamma}} \in \R$ and
$\tilde{\gamma}$ is conjugate in $SL_2(\R)$ to
$\left(\begin{array}{cc} t_{\tilde{\gamma}} & 0
\\ 0 & t^{-1}_{\tilde{\gamma}} \end{array} \right)$). Thus,
\begin{equation}\label{E:LC}
\Q \cdot L(M) = \Q \cdot \{ \:  \log \vert t_{\tilde{\gamma}} \vert
\ \vert \ \gamma \in \Gamma \setminus \{ 1 \} \ \ \text{semi-simple}
\}.
\end{equation}
One of the tools for analyzing Kleinian groups developed in
\cite{MR} is based on associating to a Zariski-dense subgroup
$\tilde{\Gamma} \subset SL_2(\R)$ (or $SL_2(\C)$) the
$\Q$-subalgebra $D = \Q[\tilde{\Gamma}^{(2)}]$ of $M_2(\R)$ (or
$M_2(\C)$) spanned by the subgroup $\tilde{\Gamma}^{(2)} \subset
\tilde{\Gamma}$ generated by squares. This algebra turns out to be a
quaternion algebra whose center is the {\it trace field} $K = K_{\tilde{\Gamma}^{(2)}}$
(the subfield generated over $\Q$ by the traces
$\tr \tilde{\gamma}$ for $\tilde{\gamma} \in \tilde{\Gamma}^{(2)}$),
cf. \cite[3.1 and 3.2]{MR}; note that $K$ is automatically contained
in $D$ as for any $\tilde{\gamma} \in \mathrm{SL}_2$ we have
$$\tilde{\gamma} + \tilde{\gamma}^{-1} = (\tr \tilde{\gamma}) \cdot
I_2.$$ (The reason for passing from $\tilde{\Gamma}$ to
$\tilde{\Gamma}^{(2)}$ can be seen in the fact that the ``true"
field of definition of a Zariski-dense subgroup is generated by the
traces in the {\it adjoint} representation, cf. \cite{Vin1}.) If the
subgroup $\Gamma$ is arithmetic then $D$ is precisely {\it the}
quaternion algebra involved in its description, so one can expect
$D$ to play a significant role also in the general case. Finally, we
note that for any semisimple $\tilde{\gamma} \in
\tilde{\Gamma}^{(2)} \setminus \{ \pm 1 \}$, the subalgebra
$K[\tilde{\gamma}]$ is a maximal \'etale subalgebra of $D$ and that
for any integer $n \neq 0$ we have $K[\tilde{\gamma}^n] =
K[\tilde{\gamma}]$.

Now, let $M_i = \mathbb{H}/\Gamma_i$ $(i = 1, 2)$ be two compact
Riemann surfaces as above, and let $D_i = \Q[\tilde{\Gamma}^{(2)}_i]$ and $K_i = Z(D_i)$,
for $i = 1, 2$, be the corresponding quaternion algebra and trace field, respectively.
Assume that $M_1$ and
$M_2$ are length-commensurable. Then it follows, for example, from
\cite[Theorem~2]{PR1} that
$$
K_1 = K_2 =: K.
$$
Furthermore, (\ref{E:LC}) implies that for any $\tilde{\gamma}_1 \in
\tilde{\Gamma}^{(2)}_1,$ there exists $\tilde{\gamma}_2 \in
\tilde{\Gamma}^{(2)}$ such that
\begin{equation}\label{E:WC1}
t^m_{\tilde{\gamma}_1} = t^n_{\tilde{\gamma}_2}
\end{equation}
for some nonzero integers $m , n$, and vice versa. Then the elements
$\gamma^m_1$ and $\gamma^n_2$ are conjugate in $SL_2(\R)$, hence
$$
K[\tilde{\gamma}_1] = K[\tilde{\gamma}^m_1] \simeq
K[\tilde{\gamma}^n_2] = K[\tilde{\gamma}_2].
$$
Thus, the length-commensurability of $M_1$ and $M_2$ translates into
the following condition: {\it $D_1$ and $D_2$ have the same
isomorphism classes of maximal \'etale subalgebras that intersect
nontrivially\footnote{I.e. contain a common element $\neq \pm 1$.}
$\tilde{\Gamma}^{(2)}_1$ and $\tilde{\Gamma}^{(2)}_2$,
respectively.} On the other hand, if $M_1$ and $M_2$ are
commensurable, then $D_1$ and $D_2$ are isomorphic as $K$-algebras
(cf. \cite{MR}). So, the expected (but currently lacking much
supporting evidence) result that {\it the compact  Riemann surfaces
that are length-commensurable to a given compact Riemann surface
form finitely many commensurability classes} leads to the following
algebraic problem:

\vskip2mm

\begin{center}

$(*)$ \parbox[t]{4mm}{\ } \parbox[t]{15.7cm}{\it Let $D$ be a central quaternion algebra over
a finitely generated field $K$, and $\Gamma \subset SL(1 , D)$ be a
finitely generated subgroup Zariski-dense in $G = \mathrm{SL}_{1 ,
D}$ with the trace field $K$. Let $\gen(D , \Gamma)$ denote the
collection of classes $[D'] \in \Br(K)$ where $D'$ is a quaternion
$K$-algebra for which there exists a finitely generated subgroup
$\Gamma' \subset SL(1 , D')$ Zariski-dense in $G' = \mathrm{SL}_{1 ,
D'}$ with the trace field $K$ such that $D$ and $D'$ have the
same isomorphism classes of \'etale subalgebras that nontrivially
intersect $\Gamma$ and $\Gamma'$, respectively. Then $\gen(D ,
\Gamma)$ is finite.} 

\end{center}

\vskip2mm

\noindent (Of course, the commensurability class of a nonarithmetic
$\Gamma$ is not determined uniquely by the corresponding quaternion
algebra $D$; in fact, E.B.~Vinberg \cite{Vin2} has constructed an
infinite family of pairwise noncommensurable cocompact lattices in
$SL_2(\R)$ that are contained in $M_2(\Q)$. On the other hand, one
can consider the following stronger ``asymmetric" version of $(*)$:
{\it Let $D$, $K$ and $\Gamma$ be as in $(*)$. Define $\gen'(D ,
\Gamma)$  to be the collection of classes $[D'] \in \Br(K)$ where
$D'$ is a central quaternion division $K$-algebra with the following
property: any maximal subfield $P$ of $D$ that is generated by an
element of $\Gamma$ admits a $K$-embedding into $D'$. Then $\gen'(D
, \Gamma)$ is finite.})

\vskip1mm

We will now indicate how $(*)$ can be generalized to arbitrary
absolutely almost simple groups. For this, the relation between the
elements $\tilde{\gamma}_1$ and $\tilde{\gamma}_2$ expressed by
(\ref{E:WC1}) needs to be replaced by the notion of {\it weak
commensurability} introduced in \cite{PR1}. Now, rather than working directly with the definition given there,
it will be more convenient to start with an equivalent form of this notion that was used in \cite[2.2]{PR-Gen}.
Let $G_1 \subset \mathrm{GL}_{N_1}$ and $G_2
\subset \mathrm{GL}_{N_2}$ be two semi-simple algebraic groups
defined over a field $F$ of characteristic zero. Semi-simple
elements $\gamma_1 \in G_1(F)$ and $\gamma_2 \in G_2(F)$ are said to
be {\it weakly commensurable} if the subgroups of $\overline{F}^{\times}$
generated by their eigenvalues intersect nontrivially. Obviously,
this notion is a direct generalization of the condition
(\ref{E:WC1}); at the same time, it is equivalent to the more
technical definition given in \cite[\S 1]{PR1} that requires the existence
of maximal $F$-tori $T_i$ of $G_i$, for $i = 1, 2,$ such that
$\gamma_i \in T_i(F)$ and for some characters $\chi_i \in X(T_i),$ we
have
$$
\chi_1(\gamma_1) = \chi_2(\gamma_2) \neq 1.
$$
Furthermore, (Zariski-dense) subgroups $\Gamma_1 \subset G_1(F)$ and
$\Gamma_2 \subset G_2(F)$ are weakly commensurable if every
semi-simple element $\gamma_1 \in \Gamma_1$ of infinite order is
weakly commensurable to some semi-simple element $\gamma_2 \in
\Gamma_2$ of infinite order, and vice versa. We refer the reader to
\cite{PR1} and the survey article \cite{PR-Gen} for results
about weakly commensurable Zariski-dense subgroups,
one of which (see \cite[Theorem 2]{PR1})
states that weakly commensurable finitely generated Zariski-dense
subgroups of absolutely almost simple algebraic groups have the same
trace field (defined in terms of the adjoint representation); we only mention here that, just as in the case
of Riemann surfaces, length-commensurability of locally
symmetric spaces of simple real algebraic groups is adequately
reflected by weak commensurability of their fundamental groups
(see \cite[2.3]{PR-Gen} for a discussion). Now, we would like to
propose the following conjecture generalizing $(*)$.


\vskip2mm

\noindent {\bf Conjecture 6.1.} {\it Let $G_1$ and $G_2$ be
absolutely simple (hence adjoint) algebraic groups over a field $F$
of characteristic zero,  let $\Gamma_1 \subset G_1(F)$ be a finitely
generated Zariski-dense subgroup, and let $K = K_{\Gamma_1}$ be the
trace field\footnotemark of $\Gamma_1$. Then there exists a
\emph{finite} collection $\mathcal{G}^{(1)}_2, \ldots ,
\mathcal{G}^{(r)}_2$ of $F/K$-forms of $G_2$
such that if $\Gamma_2 \subset G_2(F)$ is a finitely generated
Zariski-dense subgroup that is weakly commensurable to $\Gamma_1$,
then it is conjugate to a subgroup of one of the
$\mathcal{G}^{(i)}_2(K)$'s \ $(\subset G_2(F))$.}

\footnotetext{I.e., the subfield of $F$ generated by the traces
$\mathrm{tr} \: \mathrm{Ad} \: \gamma$ of all elements $\gamma \in
\Gamma_1$ in the adjoint representation.}

\vskip2mm

The connection between the analysis of weak commensurability and the
study of absolutely almost simple algebraic groups having the same
isomorphism/isogeny classes of maximal tori (and hence between
Conjectures 5.2 and 6.1) is discussed in \cite[\S 9]{PR-Gen}. We
note that Weisfeiler's Approximation Theorem \cite{Weis}, in
conjunction with finiteness results for Galois cohomology of
algebraic groups over number fields (cf. \cite[\S 6]{PR1}), allows
one to prove Conjecture 6.1 for $K$ a number field - details will be
published elsewhere. No other cases have been considered so far, but
we hope that the techniques involved in the proof of Theorem
\ref{T:Al} will lead to a proof of Conjecture 6.1 for inner forms of
type $\textsf{A}_{\ell}$. (Of course, one can also consider stronger
asymmetric versions of Conjectures 5.2 and 6.1 along the lines
indicated after the statement of $(*)$.)

\vskip3mm

\noindent {\small   {\bf Acknowledgements.} We are grateful to
A.S.~Merkurjev for useful discussions. The first-named author was
partially supported by the Canada Research Chairs program and by an
NSERC research grant. The second-named author was partially
supported by NSF grant DMS-0965758, BSF grant 2010149 and the
Humboldt Foundation. During the preparation of the final version of
this paper he was visiting the Mathematics Department at Yale whose
hospitality and support are thankfully acknowledged.}

\vskip5mm

\bibliographystyle{amsplain}

\end{document}